\documentclass{amsart}
\usepackage{url}
\usepackage{amsthm}
\usepackage{amsmath}
\usepackage{amssymb}
\usepackage{mathtools}
\usepackage{amsfonts}

\usepackage{enumitem}
\usepackage{mathscinet}
\usepackage{lipsum}
\usepackage[english]{babel}

\usepackage[utf8]{inputenc}
\usepackage[autostyle]{csquotes}

\usepackage[colorlinks=true,
    linkcolor=red,
    citecolor=blue,pdftex]{hyperref}

\usepackage[english]{babel}

\usepackage{comment}

\newcommand{\wt}[1]{\widetilde{#1}}

\newcommand{\abs}[1]{\left| #1\right|}

\newcommand{\mc}[1]{\mathcal{#1}}

\newcommand{\N}{\mathbb{N}}

\newcommand{\R}{\mathbb{R}}

\newcommand{\Z}{\mathbb{Z}}

\newtheorem{thm}{Theorem}[section]
\newtheorem{definition}[thm]{Definition}
\newtheorem{remark}[thm]{Remark}
\newtheorem{theorem}[thm]{Theorem}
\newtheorem*{oseledec*}{Oseledec Theorem}
\newtheorem{cor}[thm]{Corollary}
\newtheorem{claim}[thm]{Claim}
\newtheorem{lemma}[thm]{Lemma}
\newtheorem{prop}[thm]{Proposition}

\newtheorem{question}{Question}

\makeatletter
\def\moverlay{\mathpalette\mov@rlay}
\def\mov@rlay#1#2{\leavevmode\vtop{%
   \baselineskip\z@skip \lineskiplimit-\maxdimen
   \ialign{\hfil$\m@th#1##$\hfil\cr#2\crcr}}}
\newcommand{\charfusion}[3][\mathord]{
    #1{\ifx#1\mathop\vphantom{#2}\fi
        \mathpalette\mov@rlay{#2\cr#3}
      }
    \ifx#1\mathop\expandafter\displaylimits\fi}
\makeatother

\let\ul\underline
\let\wh\widehat

\newcommand{\nocontentsline}[3]{}
\newcommand{\tocless}[2]{\bgroup\let\addcontentsline=\nocontentsline#1{#2}\egroup}

\usepackage{wasysym}

\begin{document}

\author{Snir Ben Ovadia and Jonathan DeWitt}

\title{Anosov diffeomorphisms of open surfaces}

\newcommand{\Addresses}{{
  \bigskip
  \footnotesize

  Snir Ben Ovadia, \textsc{Department of Mathematics, Pennsylvania State University, State College, Pennsylvania 16801, United States}. \textit{E-mail address}: \texttt{snir.benovadia@psu.edu}
  
  Jonathan DeWitt, \textsc{Department of Mathematics, University of Maryland, College Park, Maryland 20742, United States}. \textit{E-mail address}: \texttt{dewitt@umd.edu}
}}

\date{}

\begin{abstract}
We study the existence of Anosov diffeomorphisms on complete open surfaces. We show that under the assumptions of density of periodic points and uniform geometry that such diffeomorphisms have a system of Margulis measures, which are a holonomy invariant and dynamically invariant system of measures along the stable and unstable leaves.
\end{abstract}

\maketitle

\tableofcontents

\section{Introduction}

A diffeomorphism $f\colon M\to M$ of a closed Riemannian manifold is called Anosov if it preserves a continuous invariant splitting $TM=E^u\oplus E^s$ and $Df$ uniformly expands vectors in $E^u$ and uniformly contracts vectors in $E^s$. These bundles are called the stable and unstable bundles of $f$ and are independent of the particular choice of metric on $M$. Anosov diffeomorphisms are of perennial interest in smooth dynamics as they are one of the most fundamental examples of a dynamical system with chaotic behavior.

In this paper, we consider the existence of Anosov diffeomorphisms on non-compact Riemannian surfaces. We shall make this notion precise, but informally we mean a diffeomorphism $f\colon M\to M$ of a complete Riemannian surface, such that $f$ preserves an invariant hyperbolic splitting.
 We study such maps under two assumptions. The first assumption asserts that the periodic points of $f$ are dense, which implies that the stable and unstable manifolds of periodic points are dense. Our second assumption is that the geometry of $M$ and $f$ is \emph{uniform} in the sense explained by Definition \ref{defn:uniform_anosov_diffeo}. Before we proceed we will make precise the setting that we are considering.

\subsection{Definition of Uniform Anosov Diffeomorphisms}\label{sec:geometry}
We now formulate a notion of an Anosov diffeomorphism on an open Riemannian surface with geometric qualities that permits analysis without unnecessary technical obstacles.

\begin{definition}\label{def:anosov_diffeo}
If $f\colon M\to M$ is a diffeomorphism of a Riemannian manifold $M$, then we say that $f$ is Anosov if there is a continuous $df$-invariant splitting $TM=E^u\oplus E^s$ such that the following holds. There exist $C,\lambda>0$ such that for all $v\in E^u\setminus \{0\}$ and all $w\in E^s\setminus \{0\}$,
\[
\|d_xf^nv\|\ge Ce^{\lambda n}\|v\| \text{ and } \|d_xf^n w\|\le Ce^{-\lambda n}\|w\|,
\]
for all $n\in \mathbb{N}$.

\end{definition}

In the case that $M$ is compact, it turns out that being an Anosov diffeomorphism is independent of the metric and can be characterized without reference to any metric. In the non-compact case there is more subtlety. For this we have the notion of a uniform manifold, which is a convenient generalization of the notion of bounded geometry.

The notion of what it means for $M$ to have bounded geometry deserves some discussion (See \cite{eldering2024notions} for more detail). The usual definition of bounded geometry essentially requires two things: positive injectivity radius and bounds on the covariant derivative of the curvature tensor, which require the manifold and metric to be $C^2$. Most results concerning Anosov diffeomorphisms hold when the diffeomorphism is $C^{1+\text{H\"older}}$, so $C^2$ regularity restricts the results that one can prove. In the case of low regularity, one can instead work with \emph{uniform manifolds}, we now quote a definition from Eldering \cite{eldering2024notions}:

\begin{definition}
(Uniform manifold). We say that a complete manifold $M$ with atlas $\mc{A}=\{(\phi_i\colon U_i\to \R^n)\vert i\in I\}$ is uniform of order $k\ge 1$ if
\begin{enumerate}
\item
there exists one uniform $\delta>0$ such that for each $x\in M$ there exists a coordinate chart $\phi_i$ that covers a ball of radius $\delta$ around $x$, i.e. 
\[
B(\phi_i(x), \delta)\subset \phi_i(U_i);
\]
\item
There is one global bound $B_k$ such that all transition maps are uniformly bounded in $C^k$ norm:
\[
\text{for all }i,j\in I\colon \|\phi_j\circ \phi_i^{-1}\|_{C^k}\le B_k.
\]
\end{enumerate}
We call these charts \emph{uniform charts}. 
\end{definition}

With this notion, we can formulate our setting.

\begin{definition}\label{defn:uniform_anosov_diffeo}
 We call an Anosov diffeomorphism of a Riemannian surface satisfying the following a \emph{uniform} Anosov diffeomorphism:
\begin{enumerate}
\item  The Anosov splitting is uniformly continuous with respect to the metric.
\item  The angles between the subspaces in the Anosov splitting are uniformly bounded below.
\item  The differential of $df$ on $E^s$ and $E^u$ is bounded, cobounded, and uniformly continuous.
\item  $M$ is complete and every stable and unstable manifold is complete in the pullback metric.
\item  $M$ is a uniform manifold.
\end{enumerate}
\end{definition}
Note that a manifold supporting a uniform Anosov diffeomorphism comes with an atlas of uniform charts.

\begin{remark}\label{rem:angle_bound}
As in the compact setting, one can deduce (2)  Definition \ref{defn:uniform_anosov_diffeo} from (1), (3), (4), (5). However, we leave this as an assumption in the definition for clarity.
\end{remark}

 \subsection{Main Result}
  Our main result is that a uniform Anosov diffeomorphism has a system of ``Margulis measures" defined on its stable and unstable leaves as long as it has sufficiently many periodic points.

 \begin{thm}\label{thm:main_thm}
    Suppose that $f\colon M\to M$ is a $C^1$ uniform Anosov diffeomorphism of a complete Riemannian surface $M$ and that $f$ has dense periodic points. 
    Then there exists $h>0$ such that $M$ admits a system of Margulis measures along its unstable foliation. These are measures $\mu^u$ defined on each unstable leaf such that:
    \begin{enumerate}
        \item The measures $\mu^u$ have full support on each leaf.
        \item The measures $\mu^u$ are invariant under stable holonomies.
        \item The measures are conformally invariant: $\mu^u_x\circ f^{-1}=e^{-h}\mu^u_{f(x)}$.
        \item 
        The measures assign infinite mass to rays in periodic unstable leaves.
    \end{enumerate}
 \end{thm}

The proof of the above theorem makes use of a coding of the diffeomorphism as a shift on countably many symbols. One interesting aspect of the proof is that it gives an equivalence between Margulis measures on the shift and Margulis measures for the diffeomorphism. In general, there is not a unique Margulis measure and any such measure is determined by a choice of a harmonic function, as described in Section \ref{sec:margulis_measures} and in particular Corollary \ref{cor:margulis_measures_locally_holonomy_invariant}, where Theorem \ref{thm:main_thm} is concluded. It is interesting that this correspondence between measures for the diffeomorphism and the shift holds both in the cases where the shift is transient and recurrent.

\subsection{Motivation}
A major motivation for this work and a possible source of an example of an Anosov diffeomorphism of an open surface is the work of Rodriguez Hertz, Rodriguez Hertz, and Ures.
In \cite{hertz2008partial}, those authors undertook to classify the $3$-manifolds that admit a conservative, non-ergodic, partially hyperbolic diffeomorphisms. Although they do not complete the classification of such diffeomorphisms, they obtain a trichotomy in \cite[Thm~1.8]{hertz2008partial}. One possibility in the trichotomy stipulates that the complement of the open accessibility classes forms an invariant lamination without compact leaves but with periodic boundary leaves. On a periodic boundary leaf, the partially hyperbolic dynamics restricts to an Anosov diffeomorphism of an open surface satisfying the hypotheses we consider in this paper. It remains unknown whether this part of the trichotomy can occur.

In fact, a closely related question motivated by their work has been stated in the literature. 
\begin{question}\cite[Question 3.13]{carrasco2018partially}\label{question:main_question}
Let $L$ be a complete immersed surface in a $3$-manifold, such that there is an Anosov dynamics on $L$ where
\begin{enumerate}
\item
each stable and unstable manifold is complete, and the angles between stable and unstable manifolds are bounded;
\item
periodic points are dense with the intrinsic topology; and
\item
the stable and unstable manifold of each periodic point are dense in $L$ with the intrinsic topology.
\end{enumerate}
Is $L$ the $2$-torus?
\end{question}
\noindent While our results give only a partial answer this question, developing tools to address this question is one of the main motivations for our work as the construction of Margulis measures can help address this problem.

Hiraide gave an elegant proof that a codimension 1 Anosov diffeomorphism is a diffeomorphism of a torus using Margulis measures \cite{hiraide2001simple}. See \cite{buzzi2022dichotomy} for a more recent construction and application to Anosov flows. Here we illustrate their use with the following application.

\begin{thm}\label{thm:main}
Suppose that $f\colon M\to M$ is a $C^1$ uniform Anosov diffeomorphism of a surface that admits a system of Margulis measures as in Theorem \ref{thm:main_thm} giving infinite mass to each unstable leaf. Then $M$ is closed.
\end{thm}
In particular, if one could construct such Margulis measures for invariant leaves of the lamination appearing in \ref{question:main_question}, it would give an affirmative answer to the question. 

\subsection{Examples of Anosov Diffeomorphisms of Open Surfaces and Manifolds} Let us consider some common examples of Anosov diffeomorphisms on open manifolds so that we can see the importance of the metric.

Let us first consider a simple construction that illuminates the role the metric plays. Consider a translation $T\colon \R\to \R$. As $T$ has no recurrence, it is straightforward to find a metric $g$ on $\R$ such that with respect to this metric $\|DT\|>c>1$: for example, take $g=e^{-t}dx^2$. However, for any metric on $\R$ for which $DT$ is uniformly bounded below, $\R$ cannot be complete as $T$ would then have a fixed point. Of course, this map isn't ``Anosov" in the sense that it has a splitting into two bundles, but it shows that there are issues that arise when one tries to make a map uniformly hyperbolic by changing the metric.

Let us now consider examples on surfaces.
There are several constructions of Anosov diffeomorphisms on $\R^2$ in the sense of Definition \ref{def:anosov_diffeo}. These examples differ in terms of what types of metrics are put on $\R^2$. The most obvious is a hyperbolic linear map, for which the metric has bounded geometry. However, there are some substantially more interesting examples. In 1973, White \cite{white1973anosov} showed that there is a complete metric and an invariant splitting of $T\R^2$, which makes the translation $(x,y)\mapsto (x+1,y)$ into an Anosov diffeomorphism. In White's example, the unstable foliation is comprised of a translation invariant collection of Reeb components. 
In addition, the metric that White constructs has unbounded curvature. 
In 1977, Mendes \cite{mendes1977anosov} studied Anosov diffeomorphisms of $\R^2$ and conjectured that any Anosov diffeomorphism of $\R^2$ with respect to a complete metric is topologically conjugate either to a translation or a hyperbolic linear map. 
For about forty years no new work on the conjecture appeared. However, recently it was shown that the Mendes conjecture was not correct.
 Matsumoto showed that there exists an Anosov diffeomorphism of $\R^2$ with respect to a complete metric that is not conjugate to either of these examples \cite{matsumoto2021example}. 
However, work of Groisman and Nitecki shows that the conjecture is true if one assumes that $f$ is the time one map of a flow \cite{groisman2021foliations}. 
An important aspect of this work is that it does not use any assumptions on the uniformity of the geometry of the diffeomorphisms being studied.

A final useful example to consider is an Anosov diffeomorphism of $\mathbb{T}^2$ with a fixed point $p$ deleted. This is an example where periodic points are dense and the stable and unstable foliations are minimal. 
Our theorem does not apply to this example as the metric is not complete. 
The same problem with trying to ``fix" the metric while keeping the same stable and unstable manifolds also occurs here when we consider the unstable manifolds of the deleted fixed point.

Besides surfaces, there is interest in Anosov diffeomorphisms and flows on other manifolds motivated both by Question \ref{question:main_question} and understanding Anosov dynamics more generally. There are not many papers that explicitly consider this question from the point of view of classification. In \cite{barthelme2025transitive}, a characterization is obtained showing when the lift of $\R$-covered Anosov flow is transitive \cite[Thm.~C]{barthelme2025transitive}.  In particular, this allows them to construct examples of transitive Anosov flows on open manifolds that are lifts of flows on closed manifolds \cite[Thm.~D]{barthelme2025transitive}. In addition, certain kinds of quantitative results have been proved, for example, for counting geodesics in infinite covers under the assumption of transitivity of the dynamics. See \cite{pollicott2017amenable} or  \cite{dougall2021anosov}. 
These results are also related, for example, to certainly physically motivated models. For example, in the study of the $\Z^d$-periodic Lorenz gas, the periodicity of the gas is analogous to the dynamics of a hyperbolic flow on a $\Z^d$-cover, see for example \cite{dolgopyat2022asymptotic}. In view of the particular applications mentioned above, it is quite natural to consider the existence of transitive examples, as \cite{barthelme2025transitive} does. 

The authors are not sure if the techniques here suffice to prove the analogous result for Anosov flows on non-compact $3$-manifolds. 
A typical approach to constructing Markov partitions of flows makes use of the first return map to a section. For this return map it seems likely that the non-compactness can introduce additional complications beyond those that exist for just diffeomorphisms. 
We note, however, that for some non-compact manifolds, Margulis measures may be obtained without effort. If we start with an Anosov flow $f^t$ on a compact manifold $M$, and we lift $f^t$ to an infinite cover of $M$, then the lift will have Margulis measures obtained by lifting those in the base.

\subsection{Outline of the Proof}
Our construction of the Margulis measures uses a countable Markov partition. 
Without any other assumptions, there might be little one could say about the associated countable Markov shift. However, due to the uniform geometry in this setting one is able to construct the Markov partition in such a way that allows one to make use of the fact that the Gurevich entropy of the associated shift space is finite.

In order to construct the Margulis measures, we first construct harmonic functions for the associated one-sided countable Markov shift. These functions are the eigenfunctions of an associated Koopman operator. Here there is a dichotomy in the behavior that the Markov shift can display: it can be either recurrent or transient. This depends on whether the Ruelle zeta function, which counts orbits that return to a compact set, converges or diverges at its radius of convergence.

In subsection \ref{subsec:harmonic_into_measure}, we use this harmonic function to construct a conformal family of measures on the shift space $\widehat{\Sigma}$. This family is a family of measures on symbolic unstable leaves. The construction of the harmonic function and its application  make crucial use of work of Sarig \cite{SarigNR}, Cyr \cite{VanCyrThesis}, and the first author \cite{InvFams}. In each case of the dynamical behavior---either transient dynamics or recurrent dynamics---the harmonic function is given by a different construction adapted to that behavior.

In subsection \ref{GlobCont}, we show how to push the measures just constructed on $\wh{\Sigma}$ to $M$ without sacrificing any of their properties. The most important property that must be retained is local holonomy invariance. While this property is immediate on the symbolic system, obtaining it for local stable holonomies on $M$---which need not respect the elements of the Markov partition---is more involved. Consequently, it becomes imperative that the harmonic function has been constructed correctly to ensure these properties. This challenge is similar to the challenge which Bowen and Marcus face in \cite{BowenMarcus77} when constructing an invariant measure to the horocyclic flow of an Axiom A diffeomorphism on a compact manifold. 

Once this is done, we have managed to construct the appropriate conditional measures on a certain collection of global unstable leaves, which we call proper. These are leaves that never intersect the unstable boundary of any Markov rectangle. A problem that then emerges when one tries to prove that the associated measures are holonomy invariant is that not every point in $M$ is uniquely coded. Consequently, one needs to work harder and study points that are uniquely coded to establish the appropriate result just for such proper leaves. Once this has been done, one finally has constructed the Margulis measures, however these measures may have atoms and also may assign some leaves finite volume.

The application of Hiraide's argument, which uses the Margulis measures we construct, appears in Section \ref{sec:end_of_proof}. As written, Hiraide's argument requires that the Margulis measures have infinite volume on every leaf and that they be non-atomic. Having atoms does not prevent Hiraide's argument from working, but having leaves with finite measure is a non-trivial issue. Surprisingly, it turns out that one can boostrap the missing infinite volume property from a dense subset of the leaves. One can then proceed with Hiraide's argument and use the Margulis measures to construct an explicit universal cover of $M$. Properties of this cover then imply that $\pi_1(M)$ is abelian. As this restriction narrows down what surface $M$ is to finite list, the conclusion that $M$ is compact follows by case analysis.

\noindent\textbf{Acknowledgements.} The authors are grateful to Aaron Brown, Andrey Gogolev, Federico Rodriguez-Hertz, and Amie Wilkinson for helpful discussions. The authors are also grateful to Emilio Corso and Andy Hammerlindl for comments on the manuscript.
The second author was supported by the National Science Foundation under Award No.~DMS-2202967.

\section{Markov Partition for Non-Compact Anosov Diffeomorphisms}

The purpose of this section is to construct a Markov partition for the dynamics of $f$ on the non-compact space $M$, following Bowen's construction as it appears in \cite{BowenRuelle}. The construction is general, and only requires uniformity properties. It does not rely on the fact that $\mathrm{dim}\,M=2$. 

\subsection{Basic Properties of Uniform Anosov Diffeomorphisms}

To begin we fix some notation and constants that will be used throughout the proof.

\begin{enumerate}
    \item 
    Let $f\colon M\to M$ be a uniform Anosov diffeomorphism of a Riemannian manifold $M$. Let $TM=E^u\oplus E^s$ denote the continuous $df$ invariant splitting of $TM$ into unstable and unstable bundles. Namely, there exist constants $C$ and $\lambda>0$ such that for any unit vectors $e^u\in E^u$ and $e^s\in E^s$,
    \[
    \|d_xf^ne^u\|\ge Ce^{\lambda n}\text{ and } \|d_xf^{-n} e^s\|\ge Ce^{\lambda n}.
    \]
    \item
    Let $\alpha>0$ denote the minimum angle between the stable and unstable subspaces, i.e.
    \[
    \alpha\coloneqq \inf_x\angle(E^s_x,E^u_x).
    \]
    Note that this number is positive by assumption.
    \item
    Let $M_f\coloneqq \sup_{x\in M}\{\|d_xf\|,\|d_xf^{-1}\|\}$; note $M_f<\infty$ by Remark \ref{rem:angle_bound}.
    \item
    Let $r_0$ denote the supremum of distance to the boundary of the domain of all uniform charts.
    \item
    We let $B_{\epsilon}(x)$ denote the closed ball of radius $\epsilon$ about the point $x\in M$.
\end{enumerate}

The following definitions are dynamical, standard, and appear in \cite[\textsection~3C]{BowenRuelle}. The proof of the following proposition is identical to that described in \cite[3.3]{bowen2008ergodic}.
\begin{prop}
(Smale Bracket) Suppose that $f$ is a uniform Anosov diffeomorphism. For sufficiently small $\epsilon>0$, there exists $\delta>0$ such for any $x,y$ with $d(x,y)<\delta$,  $W^u_{\epsilon}(x)\cap W^s_{\epsilon}y)$ consists of a single point denoted $[x,y]$. Further the map $[\cdot,\cdot]\colon (x,y)\mapsto [x,y]$ defined on pairs of points $(x,y)$ where $d(x,y)<\delta$ is uniformly continuous. This map is called the \emph{Smale bracket}.
\end{prop}

\begin{definition}
A subset $R\subseteq M$ is called a {\em rectangle} if it has a small diameter
relative to $L^{-1}$, $r_0$, and $\delta$, and such that 
 \[[x,y]\in R\text{ whenever }x,y\in R.\]

A rectangle $R$ is called {\em proper} if $R$ is closed and $R=\overline{\mathrm{int}(R)}$. 
\end{definition}

In our main results, we assume that the periodic points of $f$ are dense in $M$. Due to the existence of local product neighborhoods for uniform Anosov diffeomorphisms, it is straightforward to show that this implies that the stable and unstable manifolds of each periodic point are dense in $M$. To see this, consider the closure $C$ of an (un)stable leaf: Any open ball near the boundary of $C$ contains a periodic point whose stable manifold intersects $C$. This leads to a contradiction.

We also use the standard notation for the plaques of the $W^s$ foliation lying in a particular rectangle.
\begin{definition}
Let $R$ be a proper rectangle. For $x\in R$, let
\[W^s(x,R)\coloneqq\text{the connected component of }W^s(x)\cap R\text{ which contains }x,\]
\[W^u(x,R)\coloneqq\text{the connected component of }W^u(x)\cap R\text{ which contains }x.\]
\end{definition}

As rectangles are essentially ``aligned" with the stable and unstable foliations, we introduce notations for their stable and unstable boundaries, $\partial^s R$ and $\partial^u R$.
\begin{lemma}[{\cite[Lemma~3.11]{BowenRuelle}}]
Let $R$ be a closed rectangle. $R$ has boundary 
\[\partial R=\partial^s R\cup \partial^u R,\text{ where }\]
\begin{align*}
\partial^s R=&\{x\in R: x\notin \mathrm{int}(W^u(x,R)) \},\\
\partial^u R=&\{x\in R: x\notin \mathrm{int}(W^s(x,R)) \},
\end{align*}
where the interior of the local leaves refers to their relative interior. 
\end{lemma}

\begin{definition}
A {\em Markov partition} is a (possibly countable) covering $\mathcal{R}$ of $M$ by proper rectangles such that
\begin{enumerate}
	\item The rectangles have disjoint interiors, i.e.~ $\mathrm{int}(R_i)\cap\mathrm{int}(R_j)=\varnothing$ for $i\neq j$,
	\item (Markov property): $f[W^u(x,R_i)]\supseteq W^u(f(x),R_j)$ and $f[W^s(x,R_i)]\subseteq W^s(f(x),R_j)$ when $x
	\in \mathrm{int}(R_i), f(x)\in \mathrm{int}(R_j)$.
\end{enumerate}	
\end{definition}

The significance of the following lemma is in the uniformity of estimates of the Markov partition we construct, in spite of the partition not being finite. In turn, these estimates will allow us to deduce thermodynamic properties of the topological Markov shift induced by the Markov partition.

The following lemma is immediate from the uniform geometry of $M$ and shows that there exist ``uniform" $r$-dense sets of point in $M$. This lemma is where bounded geometry is used: its second conclusion is a type of uniform doubling inequality for the points in the dense set. Note that this lemma does indeed require the bounded geometry: consider a metric space that is a tree with uniformly long edges and unbounded degree.
\begin{lemma}\label{almostFinite}
Suppose that $f\colon M\to M$ is a uniform Anosov diffeomorphism.
	For all sufficiently small $r>0$, there exist $C_r\in \mathbb{N}$ and a subset $\mathcal{S}_r\subseteq M$ such that
	\begin{enumerate}
	\item For all $x\in M$ there exists $y \in\mathcal{S}_r$, such that $d(x,y)<\frac{r}{2}$,
	\item For all $ x\in M$, $\abs{ B_{3rM_f }(x)\cap \mathcal{S}_r}\leq C_r$.
	\end{enumerate}
\end{lemma}

The proof of the following proposition is the same as in \cite[Proposition~3.6]{bowen2008ergodic}. 

\begin{lemma}
(Shadowing Lemma) Given $\epsilon>0$ sufficiently small there exists $t_\epsilon>0$ such that every $t_\epsilon$-pseudo-orbit is $\epsilon$-shadowed.
\end{lemma}

We now show that there exists a Markov partition for a uniform Anosov diffeomorphism and that the associated graph has bounded degree.

\begin{theorem}\label{ourPtn}
Suppose that $f\colon M\to M$ is a uniform Anosov diffeomorphism.
	For all sufficiently small $r>0$ there exists $D_r>0$ such that $(M,f)$ admits a Markov partition $\mathcal{R}$ satisfying:
	\begin{enumerate}
	\item For all $R\in\mathcal{R}$, $\mathrm{diam}(R)\leq r$,
    \item Any $x\in M$ lies in at most $D_r$ rectangles in $\mc{R}$,
	\item For all $ R\in\mathcal{R}$, $\#\{S\in\mathcal{R}:f[R]\cap S\neq\varnothing\},\#\{S\in\mathcal{R}:f^{-1}[R]\cap S\neq\varnothing\}\leq D_r$.
\end{enumerate}

\end{theorem}
\begin{proof}

The construction of a Markov partition in \cite[Thm. 3.12]{bowen2008ergodic} applies in our setting. The construction begins by using a finite $r$-dense set in $M$, denoted by $P$ in the proof. In our case, we may use the set  $\mathcal{S}_r$ given by Lemma \ref{almostFinite}. The proof then remains essentially the same: one constructs a covering of $M$ by rectangles and then passes to a refinement of this covering that is the actual Markov partition. The only place where non-compactness might be an issue is Lemma 3.13, which shows the continuity of a particular coarse coding map; in our case local compactness suffices to make the same conclusion.
Hence the proof yields a Markov partition $\mc{R}$ with property (1). 

We now show that due to our choice of $\mathcal{S}_r$ that the partition $\mc{R}$ also has property (2).
The elements of $\mathcal{R}$ are the refinement of a cover of $M$ by proper rectangles where each rectangle corresponds to a point in $\mathcal{S}_r$ which is contained in it. Each $R\in\mathcal{R}$ is contained in a rectangle $R(x)$ for some $x\in\mathcal{S}_r$. Hence for any $S\in\mathcal{R}$ such that $f[R]\cap S\neq\varnothing$ we see that $S$ is contained within $B_{rM_f}(f(x))$. By item (2) in Lemma \ref{almostFinite}, there are at most $C_r$ rectangles corresponding to points of $\mathcal{S}_r$ which intersect $f[R]$. The Bowen refinement of $C_r$ rectangles cannot admit more than $C_r4^{C_r}$-many rectangles. Set $D_r:=C_r\cdot 4^{C_r}$. The proof that $\abs{\{S\in\mathcal{R}:f^{-1}[R]\cap S\neq\varnothing\}}\leq D_r$ is similar. 
\end{proof}

We record some additional basic properties without proof. The proofs are the same as in the compact case as they use only local uniformity of the dynamics.
\begin{prop}
Suppose that $f\colon M\to M$ is a uniform Anosov diffeomorphism. Then $f$ satisfies the Anosov shadowing lemma. Namely, for all $\epsilon>0$ sufficiently small there exists $\delta>0$ such that any $\delta$-pseudo-orbit is $\{x_n\}$ is shadowed by a genuine orbit $\{y_n\}$ and $d(x_n,y_n)<\epsilon$ for all $n\in \Z$. Further, $f$ is expansive.
\end{prop}

\begin{remark}\label{rem:expansive_partition}
As it will be needed in subsequent proofs, we choose the elements of the Markov partition to be small enough so that if $x\in R$, $y\in S$ and $R\cap S\neq \emptyset$, then $d(x,y)<\epsilon$, where $\epsilon>0$ is a threshold for expansiveness.
\end{remark}

\subsection{Background on Countable Markov Shifts}

In this section we review some definitions used for discussing countable Markov shifts. For more detailed discussion, see \cite{sarig2009lecture}.

Let $\mc{R}$ be a countable collection of symbols. And let $(\mathbf{A}_{ij})_{\mc{R}\times \mc{R}}$ be a (infinite) $\{0,1\}$-matrix. We define the countable Markov shift as the dynamical system on the set
\[
\wh{\Sigma}=\{\omega\in \mc{R}^{\Z}: A_{\omega_n\omega_{n+1}}=1 \text{ for all }n\in \Z\}.
\]
For an element $\underline{R}\in \hat{\Sigma}$ we write $R_i$ for its $i$th entry, $i\in \Z$. 

This space is topologized by the metric
\begin{equation}
d(\underline{R},\underline{S})\coloneqq \exp\left(-\min\{|i|:i\in\mathbb{Z}, R_i\neq S_i\}\right).
\end{equation}
We define the dynamics of the left shift of $\wh{\Sigma}$ by
\[
(\sigma\omega)_n=\omega_{n+1}.
\]

 For a single symbol $R\in \mc{R}$, we define the cylinder set
\[[S]\coloneqq \{\underline{S}\in\widehat{\Sigma}: R_0=S\}.\]
\begin{definition}
A {\em word of length $\ell$} is an $\ell$-tuple $\underline{w}=(w_1,\ldots,w_\ell)\in\mathcal{R}^\ell$, such that  for all $1\leq i\leq\ell-1$, $\sigma([w_i])\cap[w_{i+1}]\neq\varnothing$.
\end{definition}

If we have (possibly infinite) strings $\underline{S}$, $\underline{R}$, we denote their concatenation by juxtaposition $\underline{S}\underline{R}$. If there may be ambiguity about the intended operation, we denote the concatenation by $\underline{S}\cdot\underline{R}$.

\begin{definition}\label{def:transitivity}
	We say that the dynamics on $\widehat{\Sigma}$ is {\em transitive} if for any $R_1,R_2\in S$ there is some $n$ and $\underline{R}\in [R_1]$ such that $\sigma^n\underline{R}\in[R_2]$. Note that this implies the topological transitivity of $(\wh{\Sigma},\sigma)$.
\end{definition}

If one thinks of the matrix $(\mathbf{A}_{ij})$ as the adjacency matrix of a directed graph. Then we say that this matrix is irreducible if it is path connected and is aperiodic if there is some point which has loops based at that point with relatively prime period. The topological markov shift associated to an irreducible aperiodic graph is topologically mixing. In addition, if the graph has locally finite degree, then the associated topological Markov shift is locally compact.

The above definitions extend analogously to one sided shifts. In this paper we consider the shift $(\wh{\Sigma}_L,\sigma_R)$ of all left-infinite words. As well as the left shift on all right infinite words, which we denote $(\wh{\Sigma},\sigma_L)$.
\subsection{Induced topological Markov shift from a Markov Partition}

In this subsection we show that the Markov partition induces a transitive topological Markov shift with a bounded-to-one coding map of $M$. In addition, we show that the Gurevich entropy of the topological Markov shift is finite. 

We begin by introducing the shift space and its associated coding map. 

\begin{definition}\label{defn:coding_space}
(Coding Space) Suppose that $f\colon M\to M$ is a uniform Anosov diffeomorphism and that $\mc{R}$ is a Markov partition of $M$. Let
\[\widehat{\Sigma}\coloneqq\{\underline{R}\in\mathcal{R}^{\mathbb{Z}}\colon \text{ for all } i\in\mathbb{Z}, f[\mathrm{int}(R_i)]\cap \mathrm{int}(R_{i+1})\neq\varnothing\},\]
and define $\widehat{\pi}\colon \widehat{\Sigma}\to M$ to be the unique point in the set
\[\bigcap_{i\in\mathbb{Z}}f^{-i}[R_i].\]
\end{definition}
\noindent The proof that this definition and hence the coding associated to a Markov partition makes sense is the same as the proofs of \cite[Theorem~3.18]{bowen2008ergodic} and \cite[Lemma~4.2]{bowen2008ergodic}.

Later we will make a close study of the points in $M$ such that $\abs{\wh{\pi}^{-1}(\{x\})}=1$. Of particular importance  is the following two sets, $Y$ and $Y'$, which we will reference throughout the rest of the paper.
\begin{definition}\label{defn:Y}
Let $Y$ denote the points of $M$ whose orbit does not meet any unstable boundaries of any of the rectangles, i.e.~  \[Y\coloneqq M\setminus\bigcup_{j\in \mathbb{Z}}f^j[\bigcup_{R\in\mathcal{R}}\partial R.]\] Note that $Y$ is residual and that for $y\in Y$, $\abs{\wh{\pi}^{-1}(y)}=1$.

Let $Y'$ denote the set of points $x$ whose stable leaves never meet a stable boundary of a rectangle and whose unstable leaves never meet an unstable boundary of a rectangle. Note that the set $Y'$ is closed under the Smale bracket.
\end{definition}

One of the main technical issues we face is dealing with leaves that lie in the boundaries of rectangles. For this reason we will focus on the leaves that avoid the orbits of the boundaries as much as possible.

\begin{definition}\label{DefinitionPi}
Given $x\in M$ and $R\in \mathcal{R}$ is a rectangle containing $x$. We say that $W^u(x,R)$ is a {\em standard unstable leaf} if it is disjoint from the orbit of the unstable boundaries of the rectangles in $\mc{R}$, i.e.~
\[
W^u(x,R)\cap\left( \bigcup_{i\in \mathbb{Z}}f^{i}[\bigcup_{R\in\mathcal{R}}\partial^u R]\right) =\varnothing.
\]
\end{definition}

\noindent\textbf{Remark:} Inside each standard unstable leaf $W^u(x,R)$, the meagre set which is the complement to $Y$ is countable. The number of unstable leaves which are not standard is countable.

\begin{prop}
Let $(\wh{\Sigma},\sigma)$ be the coding of a uniform Anosov diffeomorphism $f$. Then $(\wh{\Sigma},\sigma)$ is locally compact.
\end{prop}

\begin{proof}

By Theorem \ref{ourPtn}, we see that there is a uniform number $D_r$ such that if $R$ is a rectangle then $f(R)$ and $f^{-1}(R)$ each intersect at most $D_r$ rectangles in $\mc{R}$. Hence the graph associated with $\wh{\Sigma}$ has in-degree and out-degree bounded by $D_r$. Thus $\wh{\Sigma}$ is locally compact.
\end{proof}

\begin{definition}\label{affiliated}
For $R,S\in\mathcal{R}$, we write $R\sim S$ if $R\cap S\neq\varnothing$.
\end{definition}

Note by Theorem \ref{ourPtn} that for all $R\in\mathcal{R}$, $\abs{\{S\in\mathcal{R}:S\sim R\}}\leq D_r$.

We now record the following Lemma that we will repeatedly use later in the argument.
\begin{lemma}\label{lem:coincide_codings}
Let $(\wh{\Sigma},\sigma)$ and $\wh{\pi}$ be the coding of a uniform Anosov diffeomorphism. Suppose that $x\in M$ and that $\underline{R},\underline{S}$ both code $x$. If there exist $j<k$ such that $R_j=S_j$ and $R_k=S_k$, then $\underline{R}=\underline{S}$.
\end{lemma}

\begin{proof}
Choose a point $y\in \text{int}(R_{j})\cap Y'$ and a point $z\in \text{int}(R_k)\cap Y'$. Let $\underline{y}$ and $\underline{z}$ denote the (unique) codings of $y$ and $z$. Then define new sequences:
\begin{align*}
    \underline{a}&=\underline{y}'\cdot (R_{j},R_{j+1},\ldots,R_{k-1},R_k)\cdot \underline{z}'\\
    \underline{b}&=\underline{y}'\cdot (R_{j},S_{j+1},\ldots,S_{k-1},R_k)\cdot \underline{z}',
\end{align*}
where $\underline{y'}$ and $\underline{z'}$ are truncated in the obvious way so that this makes sense and each of $\underline{a}$ and $\underline{b}$ defines an element of $\wh{\Sigma}$ by the Markov property.  Hence we may consider $a=\wh{\pi}(\underline{a})$ and $b=\wh{\pi}(\underline{b})$. Note by Remark \ref{rem:expansive_partition} that $a$ and $b$ shadow each other closely enough that expansiveness applies and we may conclude that $a=b$.  Note however that $a,b\in W^s(y)$ and $a,b\in W^u(z)$. We claim that $a,b\in Y'$. This is because points are in $Y'$ if their stable manifold avoids the orbit of the stable rectangle boundaries, which is true by the assumption that $y\in Y'$. The same consideration applies for the unstable rectangle boundaries. Thus we see that $a,b$ are uniquely coded as they lie in $Y'$, thus we must have that $\underline{a}=\underline{b}$ and hence $\underline{S}=\underline{R}$.
\end{proof}

\begin{prop}\label{bdd21}
Let $(\wh{\Sigma},\sigma)$ and $\wh{\pi}$ be the coding of a uniform Anosov diffeomorphism. Then there exists $C>0$ such that for all $x\in M$, $\abs{\wh{\pi}^{-1}(\{x\})}\le C$. In fact, 
\[
\abs{\widehat{\pi}^{-1}[\{x\}]}\leq (D_r+1)^2-1,
\]
where $D_r$ is the constant appearing in Theorem \ref{ourPtn}.
\end{prop}
\begin{proof}
Let $x\in M$ and let $P_x'$ denote $\widehat{\pi}^{-1}[\{x\}]$. Suppose for contradiction that $\abs{P_x'}\ge (D_r+1)^2$. Then there exists $n\in \N$ such that 
\[
\abs{\{(R_{-n},\ldots,R_n):\underline{R}\in P_x\}}\ge (D_r+1)^2.
\]
Let $P_n$ denote the set in the previous displayed equation. 
Note that for any $\underline{R},\underline{S}\in P_x$, that $R_n\sim S_n$ and $R_{-n}\sim S_{-n}$ as these code the same point. 
By Theorem \ref{ourPtn}, no point $x\in M$ lies in more than $D_r$ rectangles in the Markov partition, hence by the pigeonhole principle that there exist distinct points $\underline{R},\underline{S}\in P_n$ such that $R_{-n}=S_{-n}$ and $R_n=S_n$. 
 But by Lemma \ref{lem:coincide_codings}, this implies that $\underline{R}=\underline{S}$; a contradiction.
\end{proof}

The following theorem summarizes the main properties of the coding in Definition \ref{defn:coding_space}. We say that a map $\phi$ is uniformly locally H\"older continuous when there is are fixed $C,\alpha>0$ such that the restriction of $f$ to sets of uniform diameter $\delta>0$ is $\alpha$-H\"older with H\"older constant $C$.

\begin{theorem}\label{injectiveness}
Suppose that $f\colon M\to M$ is a uniform Anosov diffeomorphism and that $\mc{R}$ is a Markov partition of $M$. Let $(\wh{\Sigma},\sigma)$ and $\hat{\pi}$ be the associated coding space. Then $\wh{\pi}$ is well-defined, surjective, uniformly finite-to-one, and is uniformly locally H\"older continuous with respect to the metric $d$ on $\wh{\Sigma}$. In addition, $\widehat{\pi}\circ \sigma=f\circ\widehat{\pi}$ where $\sigma\colon\widehat{\Sigma}\to\widehat{\Sigma}$ is the left-shift.
\end{theorem}

\begin{proof}
That the map is well-defined and H\"older continuous is the same as in the compact case, so we omit this. That $\wh{\pi}$ is uniformly finite to one was checked in Proposition \ref{bdd21}.

We just check surjectivity.	Let $x\in M$; we show $x\in \widehat{\Sigma}$. Let $Y$ be as in Definition \ref{defn:Y}. As $Y$ is dense, there is a sequence $x_n\in Y$ such that $x_n$ converges to $x$.
	Let $\underline{R}^{(n)}$ be the preimage of $x_n$ under $\hat{\pi}$. Then as the Markov partition is locally finite, we may assume that $R_0^{(n)}=R$ for some fixed $R\in \mc{R}$. Thus by local compactness of $\widehat{\Sigma}$, we may assume that $\underline{R}^{(n)}\to \underline{R}$ for some $\underline{R}\in \widehat{\Sigma}$. Thus by continuity of $\hat{\pi}$, $\hat{\pi}(\underline{R})=x$. 
\end{proof}

\subsection{Gurevich Entropy is finite}

We now begin the proof that $(\widehat{\Sigma},\sigma)$ has finite Gurevich entropy. The Gurevich entropy is a generalization of topological entropy to the non-compact setting, which measures the growth rate of the number of periodic orbits of length $n$ that intersect a small neighborhood. In the classical Anosov case, one can similarly define entropy by studying returns to a small open set. One can check its equivalence to the usual definition using Bowen balls by using the shadowing property.
In the case of a countable Markov shift, as long as the dynamics are transitive, the Gurevich entropy is independent of the neighborhood used. (For a detailed discussion see \cite[\textsection3.1.3]{sarig2009lecture}).
The reason that the Gurevich entropy is finite in this case is essentially that the graph of the associated topological Markov shift has uniformly bounded degree. 

\begin{definition}\label{GurEnt}
Suppose that $(\wh{\Sigma},\sigma)$ is a transitive Markov shift on an alphabet $\mc{R}$.
Fix some $R\in \mc{R}$.
The {\em Gurevich entropy} of $\widehat{\Sigma}$ is \[h_G(\sigma)\coloneqq\limsup_{n\to\infty}\frac{1}{n}\log\#\{\underline{w}\text{ word of length }n: w_0=w_{n-1}=R\}\in[0,\infty].\]
\end{definition}

To apply this definition we must first show the following proposition.

\begin{prop}
Suppose that $f\colon M\to M$ is a uniform Anosov diffeomorphism, then the associated topological Markov shift $(\widehat{\Sigma},\sigma)$ is transitive. 
\end{prop}
\begin{proof}
We check the condition in Definition \ref{def:transitivity}. Hence it suffices to show that for fixed rectangles $R$ and $S$ that there exists a point whose orbit goes from the interior of $R$ to the interior of $S$ forward in time. Let $p\in \mathrm{int}(R)$ be a periodic point. Then as the unstable leaf of $p$ is dense in $M$, there exists $x\in \mathrm{int}(S)\cap W^u(p)$. Hence there exists $n\ge 0$ such that $f^{-n}(x)\in \mathrm{int}(R)$. The conclusion follows by considering the coding of $x$.
\end{proof}

\begin{cor}
The Gurevich entropy of $\widehat{\Sigma}$ is finite. 	
\end{cor}
\begin{proof}
Fix $R\in\mathcal{R}$. By Theorem \ref{ourPtn}, for all $n\geq0$, 
\[
\abs{\{\underline{w}\text{ word of length }n: w_1=w_n=R\}}\leq D_r^n,
\]
whence $h_G(\sigma)\leq \log D_r<\infty$.
\end{proof}

\section{Margulis Measures through Countable Markov Shifts}\label{sec:margulis_measures}

For an Anosov diffeomorphism of a closed manifold, the measure of maximal entropy is unique and has a number of special properties. Some of the most important properties are that its conditional measures along unstable manifolds are conformally invariant under the dynamics and holonomy invariant by stable holonomies. See, for example, the discussion in \cite{katok1997introduction}. Whereas the conditional of a measure along a foliation is typically not well-defined at every point, the MME gives a system of measures that are well defined along leaves and under local holonomies are invariant.

Because there is no reason why a finite measure MME should exist for a non-compact system, our goal is to produce a system of measures along unstable leaves that have the same properties that the conditional measures of the MME have in the compact case. We will refer to this system of measures as \emph{Margulis measures} to emphasize these properties they share with the MME in the finite dimensional setting.

In this section we construct a family of Margulis measures, i.e.~a family of leaf measures on unstable leaves, which remain invariant up to a fixed multiplicative constant when pushed by the dynamics and which are invariant under holonomies. We wish to construct these measures on $M$ without assuming the existence of a measure of maximal entropy (nor the compactness of $M$). We prove that the Margulis measures are fully supported on unstable leaves, are finite on local unstable leaves, and are infinite on global unstable leaves of periodic points. 

The construction of the Margulis measures in this section (Definition \ref{theFamMu}) is general and also applies to codings of non-uniformly hyperbolic systems, such as those studied in \cite{Sarig,SBO,LifeOfPi}.

\subsection{Coding and Associated One Sided Shifts}
We will fix for the rest of the section a Markov partition $\mc{R}$ and associated coding $\hat{\pi}\colon \widehat{\Sigma}\to M$ as in Theorem \ref{injectiveness}. In our construction, we will also work with two kinds of local-unstable sets, stable sets in $\widehat{\Sigma}$ as well as local unstable manifolds in $M$. The dynamics on the space of local unstable sets in $\widehat{\Sigma}$ is given by the dynamics on an associated $1$-sided shift, which we now describe.

\begin{definition}
(One sided shifts) Define
\[\widehat{\Sigma}_L\coloneqq \{(R_i)_{i\leq 0}:(R_i)_{i\in\mathbb{Z}}\in \widehat{\Sigma}\},\]
The space $\widehat{\Sigma}_L$ is endowed with dynamics of the {\em right-shift} $\sigma_R:\widehat{\Sigma}_L\to \widehat{\Sigma}_L$ defined by
\[ (\sigma_R(\underline{R})_i)_{i\leq 0}=(R_{i-1})_{i\leq 0}.\]

\noindent Define
\[\widehat{\Sigma}_R\coloneqq \{(R_i)_{i\ge 0}:(R_i)_{i\in\mathbb{Z}}\in \widehat{\Sigma}\},\]
The space $\widehat{\Sigma}_R$ is endowed with dynamics of the {\em left-shift} $\sigma_L$ 

\end{definition}
\noindent Conceptually, the right shift on $\wh{\Sigma}_L$ is like applying $f^{-1}$ on the set of local unstable leaves. 
We will also use the following maps, which relate $\widehat{\Sigma}$ and $\widehat{\Sigma}_L$. 
\begin{definition}\label{defn:local_unstable_sets}
Let $\tau\colon\widehat{\Sigma}\to \widehat{\Sigma}_L$ be the map $\tau((R_i)_{i\in\mathbb{Z}})= (R_i)_{i\leq0}$. Note that for a word $\underline{R}\in \widehat{\Sigma}_L$, that $\tau^{-1}[\{\underline{R}\}]$ is a (symbolic) local unstable set in $\widehat{\Sigma}$. Hence we will also denote this set by $W^u_0(\underline{R})$.  
\end{definition}

Next we will have the symbolic stable holonomies between local unstable sets.

\begin{definition}
Given a cylinder set $[R]\subset \wh{\Sigma}$, for any two points $\underline{R},\underline{S}\in[R]\subset \widehat{\Sigma}_L$, we define the local symbolic stable {\em holonomy map}  $\widehat{\Gamma}_{\underline{R}\underline{S}}\colon W^u_0(\underline{R})\to W^u_0(\underline{S})$ by $$
\widehat{\Gamma}_{\underline{R}\underline{S}}((\ldots,R_{-1},R_0)\cdot(x_0,x_1,\ldots))\coloneqq (\ldots,S_{-1},R_0)\cdot(x_0,x_1,\ldots) .
$$
\end{definition}

Let $C(\wh{\Sigma})$ denote the space of continuous real valued functions on $\wh{\Sigma}$.

\begin{definition}
The {\em Ruelle operator} $L_0\colon C(\widehat{\Sigma}_L)\to C(\widehat{\Sigma}_L) $ is defined by $$(L_0 h)(\underline{R})=\sum_{\sigma_R\underline{S}=\underline{R}}h(\underline{S}).$$
\end{definition}

As in the compact case, one finds constructs a Margulis measure by first finding an eigenvector of the operator $L_0$.

\begin{definition}\label{harmonic}
A function $\psi\colon \widehat{\Sigma}_L\to\mathbb{R}^+$ is called {\em harmonic} if
\begin{enumerate}
\item 
$\log\psi$ is uniformly continuous,
\item $L_0\psi=e^{h_G(\sigma)}\psi$.
\end{enumerate}
\end{definition}

\subsection{A Harmonic Function on the topological Markov shift}\label{SigHarSubsec}

In the classical approach to the construction of the MME for an Anosov diffeomorphism of a compact manifold, one first constructs a harmonic function of a particular Ruelle operator by exhibiting that the operator has spectral gap. Similarly, we will construct the Margulis measures by using a harmonic function for the Ruelle operator. However, in the non-compact case when the topological Markov shift is transient the argument from the compact case has a number of complications. \emph{A priori} there is no obvious reason why any harmonic functions should exist. For example, in the non-compact case the Ruelle operator  may not have a spectral gap nor need it be exponentially mixing. Hence to obtain a harmonic function more delicate work is required. Fortunately, work of Cyr \cite{VanCyrThesis} and Shwartz \cite{schwartz2019thermodynamic} shows that such harmonic functions do exist, and in the transient case there may be many of them. There is then another potential issue: the structure of the harmonic function might have little relation to the structure of the topological Markov shift and hence the associated dynamics on the coded space. This might cause problems when we attempt to turn information about holonomy invariance in the shift space into holonomy invariance on the manifold $M$. Fortunately, the way that Cyr constructs a harmonic function by counting intersections of stable sets in the topological Markov shift has a suitable dynamical structure. Later, in Claim \ref{claim:cylinder_measure}, we will use this specific form of the harmonic functions to verify the holonomy invariance of the Margulis measures by relating these intersections in $\widehat{\Sigma}$ to the intersections of a curve in $M$ with a long stable manifold.

Let us now contrast this with a more classical approach to the construction of Margulis measures such as Hiraide considers \cite{hiraide2001simple}. To construct the Margulis measures along a piece of unstable manifold $W^u_{loc}(x)$, Hiraide takes a piece of stable manifold $W^s_L(y)$ of length $L$ and places a $\delta$-mass at each intersection point of $W^s_{L}(y)$ with $W^u_{loc}(x)$. Then Hiraide normalizes and takes a weak* limit. Broadly, this is similar to the construction of Cyr. However, in the non-compact setting, unless one chooses the leaf $W^s_L(y)$ carefully, problems may arise. For example, in our setting it is possible that $W^s(y)$ might intersect $W^u_{loc}(x)$ only finitely many times, which would result in a fully supported measure.

We now proceed with the construction of the harmonic functions. We will use the same cylinder set notation for the elementary cylinders in $\wh{\Sigma}_L$.
\begin{definition}
Given $R\in\mathcal{R}$, we write $[R]\coloneqq \{\underline{R}\in\widehat{\Sigma}_L:R_0=R\}$ when it is clear from context that $[R]$ is a subset of $\widehat{\Sigma}_L$. The collection $\{[R]\}_{R\in\mathcal{R}}$ is called the set of {\em partition sets}.
\end{definition}

\begin{definition}\label{rec}
We say that $\widehat{\Sigma}_L$ is {\em recurrent} if for some $R\in\mathcal{R}$, 
\begin{equation}\label{eqn:recurrence}
\sum_{n\geq 0}e^{-nh_G(\sigma)} \#\{\underline{w}\text{ word of length }n: w_0=w_{n-1}=R\}=\infty,
\end{equation}
and otherwise we say that it is {\em transient}. Definition \ref{rec} is well-defined due to the transitivity of the dynamics on $\widehat{\Sigma}$, hence it is equivalent to say that $\hat{\Sigma}$ is recurrent if \eqref{eqn:recurrence} holds for all $R\in \mc{R}$. 

\end{definition}

Before we state the next theorem we introduce additional notation. For an element $R\in \mc{R}$ by $R\to S$ we mean the set of all $S\in \mc{R}$ that are accessible from $R$ in a single step. We denote by $R\xrightarrow[]{p}S$ the set of all $S\in \mc{R}$ that are accessible from $R$ in exactly $p$ steps. Below the $\pm$ decoration on $\underline{R}^{\pm}$ is meant to emphasize that $\underline{R}^{\pm}$ is a two-sided word.

We will use both local stable and local unstable sets in $\hat{\Sigma}$ and count intersections between them. We already introduced the symbolic local unstable sets in Definition \ref{defn:local_unstable_sets}, so we now introduce the local stable sets.  For $n\ge 0$, and $\underline{a}\in \widehat{\Sigma}$, we have local stable sets of different diameters. Set 
\[
W^s_n(\underline{a})\coloneqq\{\underline{R}^\pm\in \widehat{\Sigma}:(R^\pm_{k+n})_{k\geq0 }=(a_{k+n})_{k\geq0} \}=\sigma^{-n}(W^s_0(\underline{a})).
\]
Notice, for every $n\geq 1$, that $\abs{ W^u_0(\underline{R})\cap W^s_n(\underline{a})}$ is finite and independent of $\underline{R}\in[R_0]\subset \widehat{\Sigma}_L$. Furthermore, the set  $\bigcup_{n\geq 0} W^u_0(\underline{R})\cap W^s_n(\underline {a})$ is dense in $W^u_0(\underline{R})$. 

Set 
\[
Z_n(R,\underline{a})\coloneqq \abs{W^u_0(\underline{R})\cap W^s_{n}(\underline{a})}.
\]
for $\underline{R}\in[R]$.

We also define the notation for $R,S\in \mc{R}$. We let $Z'_n(R,S)$ denote the number of valid words of length $n$ of the form $(R,\ldots,S)$.

\begin{theorem}[Cyr, Sarig]\label{SigHar}
Suppose that $(\wh{\Sigma},\sigma)$ is a locally finite, transitive, topological Markov shift with finite Gurevich entropy. 
There exists a positive function $\psi\colon \mc{R}\to \R$ such that for all $R\in \mc{R}$
\[\psi(R)=e^{-
h_G(\sigma)}\sum_{R\to S}\psi(S)
.\]
The function defined by $\underline{R}\mapsto \psi(R_0)$ then defines a harmonic function on $\wh{\Sigma}_L$ that we also denote by $\psi$.
\end{theorem}
\begin{proof}

There are two cases depending on whether $\wh{\Sigma}$ is transient or recurrent. Each involves an explicit construction of the harmonic function as a particular limit. We begin with the recurrent case, which is due to Sarig.

If $\wh{\Sigma}$ is not topologically mixing, we can pass to a power of $\sigma$, $\sigma^k$, so that $\sigma^k$ leaves invariant a finite decomposition of $\wh{\Sigma}$ into components each of which is topologically mixing under $\sigma^k$. Hence we may suppose in the sequel that $k=1$ so that $\sigma$ has a fixed point $\underline{a}$ with initial symbol $a_0$.

Fix a periodic point $\underline{a}$. In \cite{SarigNR}, Sarig shows that for a recurrent, topologically mixing topological Markov shift, there is a sequence $n_k\to +\infty$ such that for each $R\in \mc{R}$, that 
\[
\psi(R)=\lim_{k\to \infty} \frac{\sum_{i\le n_k} e^{-ih_G(\sigma)}Z_i'(R,a_0)}{\sum_{i\le n_k} e^{-ih_G(\sigma)}Z_i'(a_0,a_0)},
\]
converges, and this limit is positive and finite for each $R$. Although Sarig does not write this exact expression, it is equivalent to the quantity $\nu_n^b(X)$ defined at the top of page 17 of that paper in the case that $\phi=0$, which is evident after noting that $L^k_0(1_{[B]})(x)$ in that paper counts the number of paths from $B$ to $x_0$. The surrounding discussion shows the additional claim that $\psi$ is positive and finite. That the function $\psi$ is harmonic is immediate by using the assertion that $L^*_0\nu=e^{h_G(\sigma)}\nu$ on the same page and writing out the definitions.

We now consider the transient case, where we follow \cite[\textsection~4.2.2]{VanCyrThesis}. As in the case of Sarig's work, Cyr works with a transitive, irreducible, topologically mixing topological Markov shift. As before, we assume that we have a word $\underline{a}$ that is fixed by $\sigma$. Cyr begins by fixing a word $\underline{\omega}\in \wh{\Sigma}$ such that for $i>j\ge 0$, $\omega_i\neq \omega_j$. Then Cyr shows that there exists a sequence $n_k$ such that for each $R\in \mc{R}$ the following limit exists and is positive and bounded:
\[
\psi'(R)= \lim_{k\to \infty} \frac{\sum_{i=0}^{\infty} e^{-ih_G(\sigma)} Z_i'(R,\omega_{n_k})}{\sum_{i=0}^\infty e^{-ih_G(\sigma)}Z_i'(a_0,\omega_{n_k})}.
\]
This definition is the quantity $f_{\infty}(R)$, which is defined at the top of p.~100 in \cite{VanCyrThesis}. The function $\psi'$ is harmonic, which is immediate from Theorem 4.3 in \cite{VanCyrThesis} as may be checked by writing out the definition of $\phi$-conformal applied to a cylinder $[R]$.
\end{proof}

\subsection{Margulis Measures as Conditionals of Generalized Measure of Maximal Entropy}\label{subsec:harmonic_into_measure}

In this section we construct a family of measures on unstable leaves that are invariant by the dynamics up to a multiplicative constant related to the entropy. When an Anosov diffeomorphism such as ours admits a finite measure of maximal entropy, the conditional measures of the MME can satisfy this property. Even when the system need not admit an MME, one can still define the conditional leaf measures and integrate them in a way which yields an invariant measure that may be infinite or even non-conservative. Hence, the measures we construct can be thought of as the conditionals of a generalized MME.

\noindent We may now state the existence of invariant measures in this setting.

\begin{theorem}[{\cite[Theorem~5.1]{InvFams}}]\label{forExactHolo2}
Let $(\wh{\Sigma},\sigma_R)$ be an irreducible, locally compact, topological Markov shift, $(\widehat{\Sigma}_L,\sigma_R)$ the associated $1$-sided shift, and let $\psi$ be a harmonic function that is constant on partition sets. Then there exists a family of Borel probability measures on $\widehat{\Sigma}$, $\{\widehat{p}_{\underline{R}}\}_{\underline{R}\in \widehat{\Sigma}_L}$, such that 
\begin{enumerate}
\item For all $\underline{R}\in\widehat{\Sigma}_L$, $\widehat{p}_{\underline{R}}$ is carried by, and fully-supported, on $\tau^{-1}[\{\ul{R}\}]=W^u_0(\underline{R})$. 
\item For all $\underline{R}\in\widehat{\Sigma}_L$, \[\widehat{p}_{\underline{R}}\circ \sigma^{-1}=\sum_{\sigma_R\underline{S}=\underline{R}}e^{-h_G(\sigma)+\log\psi(\underline{S})-\log\psi\circ\sigma_R(\underline{S})}\widehat{p}_{\underline{S}},\]

\item For all $\underline{R}, \underline{S}\in\widehat{\Sigma}_L$, $\widehat{p}_{\underline{R}}\circ \widehat{\Gamma}_{\underline{R}\underline{S}}^{-1}= \widehat{p}_{\underline{S}} $.
\end{enumerate}
\end{theorem}

Items (1) and (2) appear explicitly in \cite[Theorem~5.1]{InvFams}, where the full support property is an immediate consequence of item (2): $\widehat{p}_{\underline{R}}([R_0,a_1,\ldots ,a_{n-1}])=e^{-nh_G(\sigma)}\frac{\psi(R_0)}{\psi(a_{n-1})}>0$. Item 
(3) is a consequence of the following Corollary \cite[Corollary~5.2]{InvFams}.

\begin{cor}\label{theFamMuHat}
Under the assumptions of Theorem \ref{forExactHolo2}, 
define the family $\{\widehat{\mu}_{\underline{R}}\coloneqq\psi(\underline{R})\cdot\widehat{p}_{\underline{R}}\}_{\underline{R}\in\widehat{\Sigma}_L}$. These measures then satisfy that:
\begin{enumerate}
\item For all $\underline{R}\in\widehat{\Sigma}_L$, $\widehat{\mu}_{\underline{R}}$ is carried by, and fully-supported on, $\tau^{-1}[\{\underline{R}\}]=W^u_0(\underline{R})$,
\item For all $\underline{R}\in\widehat{\Sigma}_L$, $\widehat{\mu}_{\underline{R}}\circ \sigma^{-1}=e^{-h_G(\sigma)}\sum_{\sigma_R\underline{S}=\underline{R}}\widehat{\mu}_{\underline{S}}$,

\item For all $\underline{R}, \underline{S}\in\widehat{\Sigma}_L$, $\widehat{\mu}_{\underline{R}}\circ \widehat{\Gamma}_{\underline{R}\underline{S}}^{-1}= \widehat{\mu}_{\underline{S}} $.
\end{enumerate}
\end{cor}

\begin{definition}[Local Margulis measures]\label{theFamMu}
Let $\{\widehat{\mu}_{\underline{R}}\}_{\underline{R}\in\widehat{\Sigma}_L}$ be the family of finite Borel measures on $M$ given by Corollary \ref{theFamMuHat}. We define for $\underline{R}\in \wh{\Sigma}_L$ the family of finite Borel measures on $M$,
\[\mu_{\underline{R}}=\wh{\mu}_{\underline{R}}\circ\wh{\pi}^{-1}.\]
The family $\{\mu_{\underline{R}}\}_{\underline{R}\in\widehat{\Sigma}_L}$ is called the {\em local Margulis measures}.
\end{definition}

\begin{definition}
Given $R\in\mathcal{R}$ for any two points $x,y\in 
R$, we define the local stable  {\em holonomy map}  $\Gamma_{x,y}\colon W^u(x,R)\to W^u(y,R) $
by \[\Gamma_{x,y}(z)\coloneqq [y,z].\]
\end{definition}

Note that the above definition only applies to the segment of a local unstable manifold lying within a fixed Markov rectangle $R$; substantial effort will be dedicated to extending holonomy invariance from these holonomies to the full stable holonomies on $M$, i.e.~those that do not respect the rectangles.

\begin{theorem}\label{propsOfMargulis} Given the Margulis measures $\{\mu_{\underline{R}}\}_{\underline{R}\in\widehat{\Sigma}_L}$ from Definition \ref{theFamMu},
\begin{enumerate}
\item For all $ \underline{R}\in\widehat{\Sigma}_L$, $\mu_{\underline{R}}$ is carried by, and fully-supported on, $W^u(x,R_0)$, for any $x\in \widehat{\pi}[W^u_0(\underline{R})]$,
\item For all $ \underline{R}\in\widehat{\Sigma}_L$, $\mu_{\underline{R}}\circ f^{-1}=e^{-h_G(\sigma)}\sum_{\sigma_R\underline{S}=\underline{R}}\mu_{\underline{S}}$,
\item For all $\underline{R}, \underline{S}\in\widehat{\Sigma}_L$, $\mu_{\underline{R}}\circ \Gamma_{x,y}^{-1}= \mu_{\underline{S}} $ for any $x\in \widehat{\pi}[W^u_0(\ul{R})]$, $y\in \widehat{\pi}[W^u_0(\ul{S})]$.
\end{enumerate}
\end{theorem}
\begin{proof}\text{ }
\begin{enumerate}
\item Fix some point $x$ in a rectangle $R\subset M$, and let $U$ be an open subset of $W^u(x,R)$ containing $x$. Let $T$ be a strip defined by
\[
\bigcup_{y\in U} W^s(y,R).
\]
Let $Y$ be as in Definition \ref{defn:Y}. By the density of $Y$, there exists a point $z\in \mathrm{int}(T)\cap Y$. Consider now the point $w= [x,z]\in W^u(x,R)$. For each $i\ge 0$, let $Z_i$ be the unique element of $\mathcal{R}$ such that $f^i(z)\in \mathrm{int}(Z_i)\cap Y$. Let $n_z\in \N$ be sufficiently large that $f^{-n_z}[W^u(f^{n_z}(w),Z_{n_z})]\subseteq T\cap U$. Then,
$$\mu_{\underline{R}}(U)\geq \widehat{\mu}_{\underline{R}}([Z_1,\ldots, Z_{n_z}])>0,$$
since $\widehat{\mu}_{\underline{R}}$ is fully supported by Theorem \ref{forExactHolo2}.
\item This follows from item (2) in Corollary \ref{theFamMuHat} by using that $\widehat{\pi}\circ \sigma=f\circ \widehat{\pi}$.
\item 
The conclusion is immediate when we note that the symbolic and topological holonomies are intertwined by $\widehat{\pi}$, i.e.~ $\Gamma_{x,y}\circ\wh{\pi}=\wh{\pi}\circ \wh{\Gamma}_{\underline{R}\underline{S}}$ on the set $W^u_0(\ul{R})$.
\end{enumerate}
\end{proof}

We have now verified the fundamental properties of the Margulis measures on local unstable leaves contained inside of Markov rectangles. We now wish to extend these measures to global unstable leaves.

Note that any element $\underline{R}\in \wh{\Sigma}_L$ is associated to a unique global unstable manifold in $M$: $W^u(x)=W^u(y)$ for any $x,y\in \hat{\pi}(W^u_0(\ul{R}))$. We denote this unstable manifold by $W^u(\underline{R})$. (Note that this is distinguished from the combinatorial local stable manifold by the lack of the $0$ subscript). Define $W^s(\ul{R})$ analogously for $\ul{R}\in \widehat{\Sigma}_R$. This unstable manifold is the image of a symbolic  unstable manifold. Hence one would like to define the Margulis measures on $W^u(\underline{R})$ by pushing forward the measure from the combinatorial unstable manifold to $M$.  In order to do this, one must at least know that the resulting measure is locally finite. This is guaranteed by the following claim.

\begin{claim}\label{finiteOnArcs}
Let the set $\{\mu_{\underline{S}}\}_{\underline{S}\in \wh{\Sigma}_L}$ of local Margulis measures be as in Definition \ref{theFamMu}. For any arc $I$ in an unstable manifold, there are finitely many $\underline{Q}\in \wh{\Sigma}_L$ such that $\mu_{\underline{Q}}(I)>0$.
\end{claim}

\begin{proof}
It suffices to show this for a the local leaf of $W^u(x,R)$ lying within a Markov rectangle. For a point $y\in M$, we can consider all of the associated codings $\ul{R}\in \wh{\Sigma}_L$ that code the past itinerary of $y$. Let this set be called $\mc{P}(y)$.  For a set $X\subset M$, we write $\mc{P}(X)=\cup_{x\in X}\mc{P}(x)$. The claim is then immediate once we know that the union of $\mc{P}(x)$ over all $y\in W^u(x,R)$ is finite.

By Lemma \ref{bdd21}, there exists a constant $C$ such that any point $x\in M$ has at most $C$ pre-images under $\wh{\pi}$. There exist a finite number of local leaves $W^u(x,T_i)$ that $W^u(x,R)$ intersects. It then suffices to show that for each $T_i$ that $\mc{P}(W^u(x,T_i))$ is finite.
Suppose that there there exist distinct codings $\ul{S}^1,\ldots,\ul{S}^n\in \wh{\Sigma}_L\in \mc{P}(W^u(x,T_i))$, where $n>C$. Let $\ul{R}$ be a coding of $x$ with $R_0=T_i$ and $\ul{R}'$ be $(R_i)_{i\ge 0}$ then by the Markov property, for each $1\le i\le n$, the word $\ul{S}^i\cdot \ul{R}'\in \wh{\Sigma}$ is valid. Moreover, each  of these words codes the point $x$ as they code the unique point in $W^s(x,T_i)\cap W^u(\ul{S})$. Thus we have found $n>C$ distinct words coding the point $x$; a contradiction. The result follows.
\end{proof}

The following corollary is then immediate from Claim \ref{finiteOnArcs}.
\begin{cor}\label{GlobMeasures}
Let the set $\{\mu_{\underline{S}}\}_{\underline{S}\in \wh{\Sigma}_L}$ of local Margulis measures be as in Definition \ref{theFamMu}.
For all $\underline{R}\in \widehat{\Sigma}_L$, the measure defined by $$\mu_{W^u(\underline{R})}\coloneqq\lim_{n\to\infty}\sum_{\sigma_R^n\underline{S}=\sigma_R^n\underline{R}}\mu_{\underline{S}},$$
is well-defined as an increasing sum of measures and is finite on any finite arc $I\subseteq W^u(\underline{R})$ where $W^u(\underline{R})$ is defined above to be the full unstable manifold of $\widehat{\pi}(\underline{R})$.
\end{cor}
\begin{proof}
The sum is increasing since $\sigma_R^n\underline{S}= \sigma_R^n\underline{R}$ implies that $\sigma_R^{n+1}\underline{S}= \sigma_R^{n+1}\underline{R}$. The limit exists because by Claim \ref{finiteOnArcs} each arc is given mass by only finitely many terms in the sum.
\end{proof}

Later we will need the following recurrence properties of $\widehat{\Sigma}_L$.

\begin{definition}
The {\em recurrent part of $\widehat{\Sigma}_L$}, which we denote $\wh{\Sigma}_L^{\#}$, is the set of words in $\wh{\Sigma}_L$ that contain some particular symbol infinitely many times.
\end{definition}

In particular, the recurrent part of $\widehat{\Sigma}_L$ contains all the periodic chains in $\widehat{\Sigma}_L$. 
Note that $\tau^{-1}[\widehat{\Sigma}_L^\#]$ is exactly the set of all codings of all points that return to some compact subset of $M$ infinitely often under the forward iteration of $f^{-1}$.

\begin{prop}\label{perLeaves}
Let $ \underline{R}\in \widehat{\Sigma}_L^\#$, and let $n_k$ be an increasing sequence such that $\sigma_R^{n_k}\underline{R}\in [R_{n_0}]$ for all $k\geq0$. Then, letting $\hat{\mu}_{\ul{S}}(1)$ denote $\int 1\,d\hat{\mu}_{\ul{S}}$,
\begin{equation}\label{infMeas}
\lim_{n\to\infty}\sum_{\sigma_R^n\underline{S}=\sigma_R^n\underline{R}} \widehat{\mu}_{\underline{S}}(1)=\infty.\end{equation}
\end{prop}
\begin{proof}
Notice that for all $\underline{S}\in \wh{\Sigma}$, by definition $\widehat{\mu}_{\underline{S}}(1)=\psi(\underline{S})=\psi(S_0)$, where $\psi$ is the harmonic function defined in Corollary \ref{SigHar}. Then since $\psi$ is harmonic, 
\[\sum_{\sigma_R^n\underline{S}=\sigma_R^n\underline{R}}\widehat{\mu}_{\underline{S}}(1)=\sum_{S\xrightarrow{n}R_{-n}}\psi(S)=(L_0^n\psi)(R_{-n})=e^{nh_G(\sigma)}\psi(R_{-n}),\]
where, as before, $S\xrightarrow[]{n}R_{-n}$ means that $R_{-n}$ is accessible from $S$ in $n$ steps. Then since $R_{-n_k}=R_{-n_0}$ for infinitely many $n_k$ increasing to $\infty$, we are done. 
\end{proof}

\subsection{Proper Global Leaves}

In this subsection we define a family of global unstable leaves with many useful properties with respect to the coding, the Margulis measures, their density in $M$, and such that this family of global unstable leaves covers the residual set $Y$' from Definition \ref{defn:Y}. These properties will later be used to extend the Margulis measures in a unique way  that makes them holonomy invariant across different elements of the Markov partition.

\begin{definition}\label{afterPi}
A chain $\underline{R}\in\widehat{\Sigma}_L$ is called {\em standard} if there exists $ x\in  \widehat{\pi}[W^u_0(\ul{R})]$ such that $W^u(x,R_0)$ is a standard unstable leaf (recall Definition \ref{DefinitionPi}).
\end{definition}

\noindent\textbf{Remark:} The set of chains in $\widehat{\Sigma}_L$ which are not standard is countable and invariant.

In the sequel points $x\in M$ that are uniquely coded will be of particular importance. For this reason, we consider the following subsets that lie above the set of uniquely coded points. Note that if the entire trajectory of a point lies within the interiors of the Markov rectangles then that point is uniquely coded.

\begin{definition}[The canonical part of the coding]\text{ }
\begin{enumerate}
\item $\widehat{\Sigma}^\circ_L\coloneqq \{\underline{R}\in \widehat{\Sigma}_L:\bigcap_{i\geq0}f^i[\mathrm{int}(R_{-i})]\neq\varnothing\}$,
\item $\widehat{\Sigma}^\circ_R\coloneqq \{\underline{R}\in \wh{\Sigma}_R:\bigcap_{i\geq0}f^{-i}[\mathrm{int}(R_i)]\neq\varnothing\}$
\item $\widehat{\Sigma}^\circ\coloneqq \{\underline{R}\in \widehat{\Sigma}:\bigcap_{i\in \mathbb{Z}}f^i[\mathrm{int}(R_{-i})]\neq\varnothing\}$.

\end{enumerate}
\end{definition}

The Markov property guarantees that $\widehat{\Sigma}_L^\circ$ is shift-invariant. In addition, $\widehat{\pi}[\widehat{\Sigma}^\circ]=Y$, since it consists of all the (unique) codings of all points which never meet the boundary of any rectangle. 

The following claim shows that if one concatenates a canonical future with a canonical past, then one obtains a word that canonically codes its image under $\wh{\pi}$.

\begin{claim}\label{LeftExtends}
For any $\underline{R}^- \in \wh{\Sigma}_L^{\circ}$ and any $\underline{R}^+\in \wh{\Sigma}_R^\circ$ with $R^+_0=R^-_0$, the word $\ul{R}^-\ul{R}^+$ is in $\wh{\Sigma}^\circ$, and hence the point coded by this word is in $Y$.
\end{claim}
\begin{proof}

Let $\mc{C}_u$ denote the portion of the curve $W^u(\underline{R}^-)$ in the Markov rectangle $R_0$. The Local product coordinates on $R_0$ make $W^u(\ul{R}^-)$ look like a horizontal line in the box $[0,1]\times [0,1]$ while the curve $\mc{C}_s$ corresponding to $W^s(\ul{R}^+)$ looks like a vertical line. Note that the assumptions on $\ul{R}^-$ and $\ul{R}^+$ imply that these vertical and horizontal lines do not coincide with the boundary of the square $[0,1]\times[0,1]$; thus we see that the intersection of $W^u(\ul{R}^-)$ and $W^s(\ul{R}^+)$ is a point in the interior of $R_0$, and the result follows.
\end{proof}

\noindent\textbf{Remark:} It is clear that $\tau[\widehat{\Sigma}^\circ]\subseteq \widehat{\Sigma}_L^\circ$. Claim \ref{LeftExtends} shows that also $\tau[\widehat{\Sigma}^\circ]\supseteq \widehat{\Sigma}_L^\circ$, hence $\widehat{\Sigma}_L^\circ=\tau[\widehat{\Sigma}^\circ]$.

\begin{prop}\label{CountablyBadChains}
The set $\widehat{\Sigma}_L\setminus \widehat{\Sigma}_L^\circ$ is countable.
\end{prop}
\begin{proof}

As there are at most countably many nonstandard chains, it suffices to show that if $\ul{R}$ is a standard chain, then $\ul{R}\in \wh{\Sigma}^\circ_L$. As a reminder, a chain $\ul{R}$ is standard if $W^u(\ul{R})$ does not intersect the orbit of the unstable boundaries of the Markov rectangles. Suppose that $\ul{R}$ is such a chain and let $\ul{Q}\in \tau^{-1}[\{\ul{R}\}]$ and $x=\wh{\pi}(\ul{Q})$. Then there is a point $y\in W^s(x,R)$ such that $W^s(y)$ does not meet the orbit of the stable boundaries. Hence we see that $[x,y]$ lies in the set $Y'$ from Definition \ref{defn:Y}. Note that if $\ul{S}$ is a preimage of $y$ under $\wh{\pi}$, then $\wh{\pi}([\ul{Q},\ul{S}])=[\wh{\pi}(\ul{Q}),\wh{\pi}(\ul{S})]$. As $[x,y]\in Y'$, we see that $[\ul{Q},\ul{S}]\in \wh{\Sigma}^{\circ}$ and so $[\ul{Q},\ul{S}]\in \wh{\Sigma}^\circ_L $.
\end{proof}

The set of chains that are most useful for us are those that both are recurrent and code points that always lie in the interior of Markov rectangles. We call such chains \emph{proper}.
\begin{definition}
The set of {\em proper chains} is $\widehat{\Sigma}_L^\star\coloneqq \widehat{\Sigma}_L^\#\cap \widehat{\Sigma}_L^\circ$.
\end{definition}

\begin{remark}
Note that $\widehat{\Sigma}_L^\star$ is shift-invariant, since both $\widehat{\Sigma}_L^\#$ and $\widehat{\Sigma}_L^\circ$ are.
\end{remark}

\begin{lemma}\label{denseStndrdSharp}
The set of proper chains $\widehat{\Sigma}_L^\star$ is dense in $\widehat{\Sigma}_L$.
\end{lemma}
\begin{proof}
By Proposition \ref{CountablyBadChains}, the complement to $\widehat{\Sigma}_L^\circ$ is countable, and $\widehat{\Sigma}_L^\#$ is uncountable in any nonempty cylinder set. Thus $\widehat{\Sigma}_L^\star=\widehat{\Sigma}_L^\#\cap \widehat{\Sigma}_L^\circ$ is dense.
\end{proof}

We will construct the Margulis measures on what we call proper global unstable leaves. These are unstable leaves that contain a point whose coding backwards in time is proper, i.e.~lies in $\wh{\Sigma}_L^{\star}$.

\begin{definition}\label{def:proper_global_leaf}
A global unstable leaf $W^u$ in $M$ is called {\em proper} if there exists a proper chain $\underline{R}\in \widehat{\Sigma}_L^\star$ and $x\in \hat{\pi}(W^u_0(\ul{R}))$ such that $W^u$ contains $x$.
\end{definition}

\subsection{Global Holonomy Invariance of the Margulis Measures}\label{GlobCont}
So far we have constructed a family of measures on proper global leaves, which are dense in the manifold. In order to extend this family of measures coherently to the whole manifold, we need to establish global continuity of these measures under stable holonomy. That is, we need to establish continuity on the manifold for holonomies that do lie within a single Markov rectangle. In order to show this, we will do what amounts to a specific calculation which is enabled by use of the explicit combinatorial formula for the harmonic function from  \textsection \ref{SigHarSubsec}. The important feature of these harmonic functions is that they are defined, in a sense, dynamically because they can be interpreted as counting the number of intersections between stable and unstable manifolds. 
Note that we only prove holonomy invariance for proper leaves. Before we give the proof, we give a brief sketch of the ideas it involves. 

The idea of the proof is as follows. Recall the construction of the harmonic functions on the shift space $\wh{\Sigma}_L$. These were constructed by counting the number of words of length $n$ between two symbols $R_1,R_2\in \mc{R}$. Recall that $Z'_n(R_1,R_2)$ counts precisely this quantity. The measure of a set $U\subset W^u(x,R)$ is then determined by looking at the preimage of this set and covering its preimage with cylinders whose measure is precisely determined by counting words passing through that cylinder.  In order to make proper use of these measures, we must relate this purely combinatorial description with the topological facts of the dynamics in the manifold.  The key observation is that if $W^u(x,R_1)$ is standard then we can estimate the quantity $Z'_n(R_1,R_2)$ by counting the intersections of $W^u(x,R_1)$ with $f^{-n}(W^s(y,R_2))$. If we choose $x$ and $y$ correctly, then $\abs{W^u(x,R_1)}\cap f^{-n}(W^s(y,R_2))$ differ by at most $2$. Hence we can estimate the measure of subsets of $W^u(x,R)$ by first covering them with intervals and then counting their intersections with $f^{-n}(W^s(y,R_2))$. The important part of this is that if $R'$ is a rectangle sitting so that stable leaves in the rectangle $R_1$ exit $R_1$ into $R'$, then if $I_1\subset W^u(x,R_1)$ is an interval and $I_2\subset W^u(x,R')$ are two holonomy related intervals, then the number of intersections between $f^{-n}(W^s(y,R_2))$ and $I_1$ should be almost equal to the number of intersections of $f^{-n}(W^s(y,R_2))$ with $I_2$. Consequently, one expects that the measures $\mu_{\ul{R}}$ to have holonomy invariance even when the holonomies pass between different rectangles. This is only possible because we have constructed the harmonic functions in terms of what is essentially dynamical information.

We remark that \cite{BowenMarcus77} similarly push forward a measure from a symbolic system to a smooth system and then must check that this pushforward has holonomy invariance. A difference between their case and the present one is that the measures constructed in are non-atomic. Hence they are able to freely discard many troublesome points such as the boundaries of rectangles \cite[Lemma~4.3]{BowenMarcus77}. On the other hand, we specifically consider the case that the Margulis measures have an atom that lies on the boundary of a rectangle. This is just one of the complications that appears in applying the strategy of Bowen-Marcus.

\begin{theorem}\label{GlobContTr}
Assume that $\widehat{\Sigma}_L$ is transient, then the Margulis measures on the proper global unstable leaves are holonomy invariant. Specifically, let $R,S\in\mathcal{R}$ be adjacent rectangles such that $L\coloneqq \partial^u R\cap \partial^u S\neq\varnothing$.  Let $\underline{R}\in[R]\cap \widehat{\Sigma}_L^\star,\underline{S}\in[S]\cap \widehat{\Sigma}_L^\star$, and let $x_R\in \widehat{\pi}[W^u_0(\ul{R})]$, $x_S\in \widehat{\pi}[W^u_0(\ul{S})]$. Let $V^s(L)$ be the saturation of local stable leaves in $R$ and $S$ of $L$, and set $I_R\coloneqq V^s(L)\cap W^u(x_R,R)$ and $I_S:=V^s(L)\cap W^u(x_S,S)$. Then $\mu_{W^u(\underline{R})}|_{I_R}=\mu_{W^u(\underline{S})}|_{I_S}\circ \Gamma$, where $\Gamma$ is the local holonomy along the stable foliation restricted to $V^s(L)$. The same holds in the case that $\wh{\Sigma}_L$ is recurrent.
\end{theorem}
\begin{proof}
The proof of the theorem is long and hence is broken into several steps with several intermediate claims. 

We begin by recalling the construction of the harmonic function $\psi$ in Theorem \ref{SigHar} for the transient case. We fixed a chain $\ul{\omega}\in \wh{\Sigma}$ such that for any $i>j\ge 0$ $\omega_i\neq \omega_j$ as well as a fixed point $a\in M$ lying in the interior of a rectangle and let $\ul{a}$ denote the coding of $a$'s past, which we write $(\ldots,a,a,a)\in \wh{\Sigma}_L$.

We now construct the sequence of stable manifolds we will use to count the intersections with local unstable leaves. Fix $k\ge 0$, fix a periodic point $x_k\in \mathrm{int}(\omega_{n_k})$ whose orbit only visits the interiors of rectangles, and denote the period of $x_k$ by $\ell_k$.

\begin{claim}\label{claim:uniquely_coded}
If $x\in W^u(x_R,R)\cap W^s(x_k)$, then $x$ has a coding in $W^u_0(\ul{R})$ and this coding is unique.
\end{claim}
\begin{proof}
We begin by showing that a coding exists. To begin, take a sequence of points $y_n\in Y'$ such that $y_n\to x$. Let $\ul{R}^{(n)}$ denote the coding of these points, which lies in $\wh{\Sigma}^{\circ}$. Then consider the points $W_n=(R_i)_{i\le 0}\cdot (R_i^{(n)})_{i\ge 0}$, which lie in $\wh{\Sigma}^{\circ}$ by Claim \ref{LeftExtends}. By the local compactness of $\wh{\Sigma}$, some subsequence converges to a point $W\in W^u_0(\ul{R})$. Note then that this point codes $x$ by continuity as each point $W_n$ codes the point $[x_R,y_n]\in W^u(x_R,R)$.

We now show that this coding is unique. Let $\ul{u}^1$ and $\ul{u}^2$ be two codings in $W^u_0(\ul{R})$ of a point $x\in W^u(x_R,R)\cap W^s(x_k)$. Note that sufficiently far in the future $\ul{u}^1$ and $\ul{u}^2$ are equal as they approach $x_k$, whose orbit always lies in the interior of Markov rectangles. As these codings both lie in $W^u_0(\ul{R})$, the agree in the past as well. Hence by Lemma \ref{lem:coincide_codings} $\ul{u}^1=\ul{u}^2$.
\end{proof}

We will now show the holonomy invariance of $\mu_{W^u(\ul{R})}$. For simplicity, we will consider in detail only the most difficult case: when $\mu_{W^u(\ul{R})}$ has an atom that lies on the common stable boundary of $R$ and another rectangle $T$ that also abuts $S$. Note that we are specializing to a specific configuration of rectangles.
This argument contains all the difficulties that occur in the remaining cases. We begin by finding a formula for a measure of a cylinder in terms of counting intersections of a stable and unstable manifold.

Previously, when we defined the harmonic function, we had the notation $Z_i'(R_1,R_2)$ that counted the number of paths between $R_1$ and $R_2$ of length $i$. Now for a cylinder $C=[C_1,\ldots,C_N]$ and $\omega_{n_k}\in \mc{R}$, we extend the definition of $Z_i'$ so that $Z_i'(C,\omega_{n_k})$ counts the number of paths from $C_1$ to $\omega_{n_k}$ of length $i$ of the form $C_1,C_2,\ldots,C_N,\ldots,\omega_{n_k}$. Note that as long as $i$ is sufficiently large this is equal to the number of paths from $C_N$ to $\omega_{n_k}$ of length $i-N+1$. 

\begin{claim}\label{claim:intersections_count_words}
Suppose that $\ul{R}\in \wh{\Sigma}_L$ and that that $C$ is a cylinder contained in $W^u_0(\ul{R})$. Then 
\[
\abs{Z_i'(C,\omega_{n_k})}=\abs{\wh{\pi}(C)\cap f^{-i}(W^s(x_k,\omega_{n_k}))}.
\]
\end{claim}
\begin{proof}
Each point $x$ in $\wh{\pi}(C)\cap f^{-i}(W^s(x_k,\omega_{n_k}))$ is uniquely coded by Claim \ref{claim:uniquely_coded}. Further, the map that carries $x$ to its coding in $W^u_0(\ul{R})$ is a surjection. This follows because by the Markov property any word $w\in Z_i(C,\omega_{n_k})$ may be extended to a two sided word $\ul{w}\in \wh{\Sigma}$ lying in the cylinder $C$ such that $\ul{w}$'s $n_k$th iterate lies in the partition set $[\omega_{n_k}]$ in the local stable manifold of the periodic point $x_k$. Hence the image of this point lies in $\wh{\pi}(C)\cap f^{-i}(W^s(x_k,\omega_{n_k}))$.
\end{proof}

An important and useful property that we will use is that the image of a cylinder $C\subset W^u_0(\ul{S})$ is a closed interval. 
\begin{claim}\label{claim:cylinder_measure}
We have the following formula for the measure of a cylinder $C$,
\[
\wh{\mu}_{\ul{S}}(C)=\lim_{k\to \infty} \frac{\sum_{i\ge 0} e^{-ih_G(\sigma)}\abs{\wh{\pi}(C)\cap f^{-i}(W^s(x_k,\omega_k))}}{\sum_{i\ge 0} e^{-ih_G(\sigma)}\abs{\widehat{\pi}(W^u_0(\ul{a}))\cap f^{-i}(W^s(x_k,\omega_{k}))}}.
\]
\end{claim}

\begin{proof}
We first will obtain a combinatorial formula for the measure of a cylinder set. 

By (2) in Corollary \ref{theFamMuHat}, which gives for a cylinder $[C_0]$ of length $1$ that:
\[
\wh{\mu}_{\ul{S}}([C_0])=\wh{\mu}_{\ul{S}}(1_{[C_0
]}\circ \sigma)=e^{-h_G(\sigma)}\sum_{\sigma_R\ul{S}=\ul{R}}\wh{\mu}_{\ul{SC_0}}(1_{[C_{0}]}\circ \sigma)=e^{-h_G(\sigma)}\wh{\mu}_{\ul{SC_0}}(1),
\]
where the last equality follows because only one term in the sum assigns any measure to the integrand.
Hence for a cylinder $C$ of length $N$, 
\[
\wh{\mu}_{\ul{S}}(C)=e^{-Nh_G(\sigma)}\wh{\mu}_{\ul{SC}}(1).
\]
Then, by definition, as the chain is transient we can rewrite this as
\begin{align*}
\wh{\mu}_{\ul{S}}(C)&=e^{-Nh_{G}(\sigma)}\hat{\mu}_{\ul{SC}}(1)\\
&=e^{-Nh_G(\sigma)}\lim_k \frac{\sum_{i\ge 0}e^{-ih_G(\sigma)}Z_i'(C_N,\omega_{n_k})}{\sum_{i\ge 0} e^{-ih_G(\sigma)}Z_i
'(\ul{a},\omega_{n_k})}\\
&=e^{-Nh_G(\sigma)}\lim_k \frac{\sum_{i\ge 0} e^{-ih_G(\sigma)}Z_{i+N}'(C,\omega_{n_k})}{\sum_{i\ge 0}e^{-ih_G(\sigma)}Z_i'(\ul{a},\omega_{n_k})}\\
&=\lim_k \frac{O(1)+\sum_{i\ge 0} e^{-ih_G(\sigma)}Z_{i}'(C,\omega_{n_k})}{\sum_{i\ge 0}e^{-ih_G(\sigma)}Z_i'(\ul{a},\omega_{n_k})},\\
\end{align*}
where the $O(1)$ is uniformly bounded because as $k\to \infty$, arbitrarily many of the finite number of terms that were gained when the sum was reindexed are zero as $\omega_{n_k}$ moves away from $\wh{\pi}(C)$ in the manifold. We now apply Claim \ref{claim:intersections_count_words} to conclude.
\end{proof}

For a rectangle $R$ and a set $U$ contained in an unstable leaf passing through $R$, we consider coverings of these sets by the images of cylinders.
\[
\mc{C}^N_R(U)=\{\text{all cylinders } C \text{ in } W^u_0(\ul{R}) \text{ of length } N \text{ such that } C\cap \wh{\pi}^{-1}[\{x\}]\neq \emptyset\}.
\]

In addition, we define the set
\[
\mc{D}^R_N(U)=\bigcup_{C\in \mc{C}^N_R(U)} \wh{\pi}(C).
\]

The following important claim shows that we can relate these coverings between different sets:
\begin{claim}\label{claim:holonomy_covering}
Suppose that $U$ and $\Gamma(U)$ are two holonomy related sets. Suppose, in addition that $S$ lies in two rectangles $R$ and $T$ and that $\Gamma(S)$ lies in a rectangle $S$.

\begin{enumerate}
\item
Then there exists $N'\ge N$ such that 
\[
\mc{D}_{N'}^S(\Gamma(U))\subseteq \Gamma(\mc{D}_N^S(U)\cup \mc{D}_N^T(U)).
\]
\item
For any segment $I$ contained in a stable leaf:
\[
\abs{\abs{U\cap f^{-n}(I)}-\abs{\Gamma(U)\cap f^{-n}(I)}}\le 2.
\]
\end{enumerate}
\end{claim}
\begin{proof}

The first claim follows because the image of a cylinder is a closed interval contained in an unstable leaf. Consequently, as the radius of these cylinders goes to $0$ as $N\to \infty$, we see that one covering is always contained in another sufficiently finer covering. The second claim is immediate because the sets $U$ and $\Gamma(U)$ are holonomy related by a local holonomy.
\end{proof}

\begin{claim}
Suppose that $x$ lies in the common stable boundary of $R$ and $T$. Then $\mu_{W^u(\ul{R})}(\{x\})=\mu_{W^u(\ul{S})}(\{\Gamma(x)\})$, where $\Gamma$ is the local stable holonomy between $I_R$ and $I_S$.
\end{claim}

\begin{proof}
Let us begin by writing the definition of the measure of $x$. By construction of $\mu_{W^u(\ul{R})}$, 
\[
\mu_{W^u(\ul{R})}(\{x\})=\mu_{\ul{R}}(\{x\})+\mu_{\ul{T}}(\{x\}).
\]
We have the following description of these measures:
\begin{align*}
\mu_{\underline{R}}(\{x\})=&\inf_N \sum_{C\in \mc{C}_N^R(\{x\})} \wh{\mu}_{\ul{R}}(C).
\end{align*}

For each $N$ that 
\begin{align*}
\mu_{\ul{S}}(\{\Gamma(x)\})&\le \sum_{C\in \mc{C}_N^S(\{x\})} \wh{\mu}_{\ul{S}}(C)\\
&=\sum_{C\in \mc{C}_N^S(\{\Gamma(x)\})}\wh{\mu}_{\ul{S}}(C)\\
&=\sum_{C\in \mc{C}_N^S(\{\Gamma(x)\})}\lim_{k\to \infty} \frac{\sum_{i\ge 0} e^{-ih_G(\sigma)}\abs{\wh{\pi}(C)\cap f^{-i}(W^s(x_k,\omega_k))}}{\sum_{i\ge 0} e^{-ih_G(\sigma)}\abs{\wh{\pi}(W^u_0(\ul{a}))\cap f^{-i}(W^s(x_k,\omega_{k}))}}.
\end{align*}
Hence we see that 
\[
\mu_{\ul{S}}(\{\Gamma(x)\})\le \lim_{k\to \infty} \frac{\sum_{i\ge 0} e^{-ih_G(\sigma)}\abs{\wh{\pi}(\mc{D}_N^S)\cap f^{-i}(W^s(x_k,\omega_k))}}{\sum_{i\ge 0} e^{-ih_G(\sigma)}\abs{\wh{\pi}(W^u_0(\ul{a}))\cap f^{-i}(W^s(x_k,\omega_{k}))}}.
\]

Now, note from the first part of Claim \ref{claim:holonomy_covering} for each $N$ there exists $N'$ such that 

\[
\mc{D}_{N'}^S(\{\Gamma(x)\})\subseteq \Gamma(\mc{D}_N^S(\{x\})\cup \mc{D}_N^T(\{x\})).
\]
By the second part of Claim \ref{claim:holonomy_covering} we then see that 
\[
\mu_{\ul{S}}(\{\Gamma(x)\})\le \lim_{k\to \infty} \frac{\sum_{i\ge 0} e^{-ih_G(\sigma)}(\abs{\wh{\pi}(\mc{D}_N^R\cup \mc{D}_N^T)\cap f^{-i}(W^s(x_k,\omega_k))}\pm 2)}{\sum_{i\ge 0} e^{-ih_G(\sigma)}\abs{\wh{\pi}(W^u_0(\ul{a}))\cap f^{-i}(W^s(x_k,\omega_{k}))}}.
\]
Note, however, that for $k$ sufficiently large that $f^{-n}(x_k,\omega_k)$ intersects none of the sets $\mc{D}_N^R,\mc{D}_N^S$ or $\mc{D}_N^T$. Thus we see in fact this stronger estimate:
\begin{equation}\label{eqn:limits_agree2}
\mu_{\ul{S}}(\{\Gamma(x)\})\le \lim_{k\to \infty} \frac{\sum_{i\ge 0} e^{-ih_G(\sigma)}(\abs{\wh{\pi}(\mc{D}_N^R\cup \mc{D}_N^T)\cap f^{-i}(W^s(x_k,\omega_k))})}{\sum_{i\ge 0} e^{-ih_G(\sigma)}\abs{\widehat{\pi}(W^u_0(\ul{a}))\cap f^{-i}(W^s(x_k,\omega_{k}))}}.
\end{equation}
Reversing the above reasoning,  we identify the limit on the right hand side of this equation as calculating the $N$th term in the infimum defining
\[
\mu_{\ul{R}}(\{x\})+\mu_{\ul{T}}(\{x\})=\inf_N \sum_{C\in \mc{C}_N^R(\{x\})} \wh{\mu}_{\ul{R}}(C)+\sum_{C\in \mc{C}_N^R(\{x\})} \wh{\mu}_{\ul{R}}(C).
\]
And hence we see that 
\[
\mu_{\ul{S}}(\{\Gamma(x)\})\le \mu_{\ul{R}}(\{x\})+\mu_{\ul{T}}(\{x\}).
\]
The reverse inequality follows along exactly the same lines and so the claim holds.

\end{proof}

The proof for the remaining cases of the relative positions of the Markov rectangles is identical, but the notation is more or less complicated depending on the particular configuration of the rectangles. One takes a cover of the set being measured by the images of cylinders and then estimates the measures of the cylinders by looking at intersections exactly as before.

The same conclusion holds if the dynamics is recurrent. The argument is conceptually identical and uses similar logic to deduce \eqref{eqn:limits_agree2}. The only difference is that we begin from a slightly different definition of the measure because the harmonic function we use for the recurrent case is different:
\[
\psi(R)=\lim_{k\to \infty} \frac{\sum_{i\le n_k} e^{-ih_G(\sigma)}Z_i'(R,a_0)}{\sum_{i\le n_k} e^{-ih_G(\sigma)}Z_i'(a_0,a_0)}.
\]
For example, using a fixed word $\underline{a}$ that lies in the interior of a rectangle with zeroth symbol $a_0$, gives that the analog of \eqref{eqn:limits_agree2} in the recurrent case would be
\[
\mu_{\ul{S}}(\{\Gamma(x)\})\le \lim_{k\to \infty} \frac{\sum_{i\le n_k} e^{-ih_G(\sigma)}\abs{\wh{\pi}(\mc{D}_N^R\cup \mc{D}_N^T)\cap f^{-i}(W^s(\ul{a},a_0))}}{\sum_{i\le n_k} e^{-ih_G(\sigma)}\abs{\wh{\pi}(W^u(\ul{a}))\cap f^{-i}(W^s(\ul{a},a_0))}}.
\]
\end{proof}

\medskip
Recall that in Corollary \ref{GlobMeasures}, we defined measures on each global leaf. However, for the non-proper leaves we will now redefine the Margulis measures by using local holonomies. The above theorem asserts that this can be done in a consistent way due to the continuity and consistency of the holonomies.

\begin{cor}\label{cor:margulis_measures_locally_holonomy_invariant}
The Margulis measures on the proper leaves extend uniquely and continuously to every global leaf $W^u$ by the following formula giving a family $\{\mu_{W^u}\}$, which we call the \emph{Margulis measures} without qualification: Given a finite arc $I\subseteq W^u$,
\[
\mu_{W^u}(I):= \lim_{\underline{R}\in\widehat{\Sigma}_L^\star: W^u(\underline{R})\to W^u}\mu_{W^u(\underline{R})}(\Gamma_{\underline{R}}[I]),\]
where $\Gamma$ is the local holonomy on stable leaves between $W^u$ and $W^u(\underline{R})$.

Further, the property $\mu_{W^u(x)}\circ f^{-1}=e^{-h_G(\sigma)} \mu_{W^u(f(x))}$ holds for the Margulis measures, and they are invariant under local stable holonomies.
\end{cor}

The Margulis measures give infinite volume to leaves through periodic points.

\begin{lemma}\label{halfArcInf}
Let $p$ be a periodic point. Fix an orientation on $W^u(p)$ so that we may unambigiously write $W^u(p)$ as the union of two rays $I^+\cup I^-$ that intersect at the point $p$. Then 
\[
\mu_{W^u(p)}(I^+)=\mu_{W^u(p)}(I^{-})=\infty.
\]
\end{lemma}
\begin{proof}
Without loss of generality, we suppose that $p$ is a fixed point. Then as the stable leaves of periodic points of $f$ are dense, there exists a periodic point $q$ in the interior of some rectangle $R$ such that $W^s(q,R)$ intersects $I^+$. Note that some power of $f$ carries $I^+$ and $I^-$ to themselves. 
Hence by iterating the dynamics, one sees that $W^u(p)$ and $I^+$ intersect infinitely many times in the rectangle $R$. 
Hence $I^+$ passes through $R$ infinitely many times and by construction of the Margulis measures each of these intersections is given mass at least $\psi(R)$ by $\mu_{W^u(p)}$. 
Hence $I^+$ has infinite measure. The same considerations apply for $I^-$.
\end{proof}

We can now conclude the proof of Theorem \ref{thm:main_thm}.

\begin{proof}[Proof of Theorem \ref{thm:main_thm}]
We have essentially finished the proof of this theorem as the measure we constructed in this section has been shown to satisfy properties (1), (2), (3), (4) required by that theorem. The final needed claim is that $h>0$. But due to the full support of the Margulis measures and the conformal invariance, that $h>0$ follows by considering the restriction of the measure to a periodic leaf in a neighborhood a periodic point.
\end{proof}

\section{Application to Anosov Diffeomorphisms on Open Surfaces}\label{sec:end_of_proof}
In this section, we give an application of the construction of the Margulis measures to Anosov diffeomorphisms on open surfaces under some assumptions. In the first subsection, we show that the Margulis measures are compatible with global product structure and give a simple description of the universal cover of $M$ when $f$ has global product structure.

\subsection{Margulis Measures and Global Product Structure}

In this section, we show that the Margulis measures give a particularly nice description of the universal cover of $M$ when it has Global product structure. 

First we recall the definition of global product structure for an Anosov diffeomorphism $f\colon M\to M$. If we let $\wt{W}^u$ and $\wt{W}^s$ denote the stable and unstable foliations of the universal cover of $M$, then we say that $f$ has \emph{global product structure} if for any $x,y\in M$, $\abs{\wt{W}^u(x)\cap \wt{W}^s(y)}=1$, i.e.~every pair of stable and unstable leaves meets in exactly one point. In the case of a surface $M$, this implies that the foliations on $\wt{M}\approx \R^2$ look like the foliations of $\R^2$ by vertical and horizontal lines.

\begin{prop}\label{prop:margulis_measures_compatible_with_product_structure}
Suppose that $f\colon M\to M$ is a uniform Anosov diffeomorphism with global product structure. Let $\mu^u$ be a system of Margulis measures from Corollary \ref{cor:margulis_measures_locally_holonomy_invariant}.
For a fixed point $p$ of $f$, we will define a map $\pi_p$ by 
\begin{equation}\label{eq:def_of_pi_p}
(x,y)\in W^s(p)\times W^u(p)\mapsto z,
\end{equation}
where $z$ is the unique point satisfying both
\begin{equation}\label{eqn:interval_measures_agree_defn}
\mu^u_{W^u(y)}([y,z])=\mu^u_{W^u(p)}([p,x])\text{ and } \mu^s_{W^s(x)}([x,z])=\mu^s_{W^s(p)}([p,y]),
\end{equation}
and such that $[y,z], [p,x]$ have the same orientation and $[x,z]$ and $[p,y]$ have the same orientation. The map $\pi_p$ is the universal cover of $M$.
\end{prop}
\begin{proof}
Let $\wt{M}$ denote the universal cover of $M$ and $\wt{W}^s$ and $\wt{W}^u$ be the lifts of the stable and unstable foliations of $f$ to $\wt{M}$. Let $P\colon \R\times \R\to \wt{M}$ be the map that shows that $\wt{W}^u$ and $\wt{W}^s$ have global product structure in the universal cover with the point $(0,0)$ mapping to some point in $\pi^{-1}(p)$. Then clearly $\pi\circ P$ is the universal cover of $M$.

It then suffices to check that this map $P$ satisfies the property in \eqref{eqn:interval_measures_agree_defn}. Certainly, due to the full support of $\mu^u$ and $\mu^s$, as long as it is defined any map $\pi_p$ satisfying \eqref{eqn:interval_measures_agree_defn} is uniquely defined up to the choice of basepoint $p$. Note now that the property \eqref{eqn:interval_measures_agree_defn} is satisfied within any foliation box due to the holonomy invariance of the Margulis measures. The global product structure shows that for any $x,y,z$ as above, we may include them into the image of a single large foliation box in the universal cover, and the result follows.
\end{proof}

The proof of Proposition \ref{prop:margulis_measures_compatible_with_product_structure} makes use of the global product structure. A posteriori, it follows that for every leaf of $W^u$ that $W^u(x)$ has infinite measure with respect to $\mu^u$. In fact, if we had known that every leaf had infinite volume a priori, then we could have deduced that there was global product structure for these foliations along similar lines. This is the argument of Hiraide \cite{hiraide2001simple} that codimension 1 Anosov diffeomorphisms have global product structure. The application in this paper could also be concluded if we had assumed instead that every leaf had infinite measure.  Note that if we tried to define the universal cover by \eqref{eqn:interval_measures_agree_defn} but every leaf did not have infinite measure, then we would not necessarily have been able to define the map: consider for instance the holonomy between an infinite measure leaf and a finite measure leaf.

We say that a foliation $\mc{F}$ of a manifold $M$ is $\R$-covered when the lift of $\mc{F}$ to the universal cover $\wt{M}$ of $M$ has leaf space isomorphic to $\R$. See, e.g.~\cite[p.~82]{fenley1994anosov}.

\begin{prop}\label{prop:margulis_measures_imply_R_covered}
Suppose that $f\colon M\to M$ is a uniform Anosov diffeomorphism and that $f$ has Margulis measures $\mu^u$ as in Theorem \ref{thm:main_thm} that assign infinite measure to each unstable leaf. Then $W^s$ is $\R$-covered.
\end{prop}
\begin{proof}
The proof follows the ideas of Hiraide. There are some minor complications only  because the measures we consider might have atoms. We can construct a covering map analogously with the construction in Proposition \ref{prop:margulis_measures_compatible_with_product_structure}. 

Take a fixed point $p$ of $f$ and leaves $W^s(p)$ and $W^u(p)$. Then we define a map $\pi\colon W^s(p)\times W^u(p)\to M$ by $\pi(z,w)=q$ exactly when $q$ is the point in $W^u(z)$ such that $\mu^u([p,w])=\mu^u([z,q])$ and this agrees with the orientation where $W^u(z)$ meets $W^s(p)$ at $z$. We must now check that there is such a point $q$. Note next that for any $z\in W^s(p)$ there is a product neighborhood containing $z$ whose $W^{u}$ plaques are arbitrarily long. Hence locally the definition of $\pi$ makes sense; what we need to do is check that we can extend this definition from an initial product neighborhood containing $p\in W^s(p)$. Fix some $r>0$ and $M>0$. Then, using that the Margulis measure of each $W^u(p)$ leaf is infinite, for any $z\in B_r(p)\subseteq W^s(p)$, we may find a product neighborhood of $z$ containing all points $q'$ with $\mu^u([w,q'])\le M$ where $w$ is the corresponding point in $W^s(p)$. These neighborhoods give an open cover of $B_r(p)$. By compactness of $B_r(p)\subseteq W^s(p)$, we can pass to a finite subcover of this set. Holonomy invariance of the Margulis measure then gives the needed conclusion on this subset. Thus the map $\pi$ is well defined.

We now check that this map defines the universal cover. We must check that $\pi$ is a continuous, surjective, local homeomorphism. Essentially all of these points follow from working in a small neighborhood of a product neighborhood $(z,w)\in W^s(p)\times W^u(p)$. Continuity is immediate from the full support of the Margulis measures. The map is a local injection because it is increasing with respect to the Margulis measures along $W^u$ leaves. Hence due to invariance of domain, it is a local homeomorphism as it is a continuous, local injection. Finally, the map is a surjection because $W^{u}(p)$ is dense by assumption as $p$ is a periodic point; surjectivity then follows from the uniform transversality of the stable and unstable foliations.
\end{proof}

Once the stable foliation is known to be $\R$-covered, one can conclude that the fundamental group of $M$ is abelian. 
Thomas Barthelme pointed out the following proof to us, which he told us was also known to Jana Rodriguez Hertz and Raul Ures.

\begin{prop}\label{prop:R_covered_implies_abelian_pi_1}
Suppose that $M$ is a closed or open surface admitting a foliation $\mc{F}$ that is $\R$-covered and has no closed leaves, such as $\mc{W}^s$. Then $M$ has abelian fundamental group. 
\end{prop}

\begin{proof}
Consider the lift of $\mc{F}$ to $\wt{M}$, $\wt{\mc{F}}$. Now consider the action of $\pi_1(M)$ on the leaf space of $\wt{\mc{F}}$. Note that the action must be free: otherwise $\mc{F}$ would contain a closed leaf. As $\pi_1(M)$ acts freely on $\R$, it follows that $\pi_1(M)$ is an Archimedean group. But by H\"older's theorem \cite[Thm.~6.10]{ghys2001groups}, Archimedean groups are abelian.  
\end{proof}

In particular Theorem \ref{thm:main} now follows.

\begin{proof}[Proof of Theorem \ref{thm:main}.]

From Proposition \ref{prop:margulis_measures_imply_R_covered} and Proposition \ref{prop:R_covered_implies_abelian_pi_1}, it follows that $\pi_1(M)$ is abelian. From the classification of open surfaces \cite[\textsection~4.2.2]{stillwell1993classical}, this implies that $M$ is either $\R^2$ or $\mathbb{S}^1\times \R$, or a M\"obius strip. In the case of $\R^2$, the result of Mendes \cite[Proposition~1.2]{mendes1977anosov} implies that $f$ has at most one non-wandering point; but this is impossible because the periodic points of $f$ are dense. If $M$ is diffeomorphic to a cylinder \cite[Lemma 4.3]{hammerlindl2019ergodicity} shows that the cylinder does not support Anosov dynamics with dense periodic points. A similar argument to the one in \cite[Lemma~4.3]{hammerlindl2019ergodicity} shows that Anosov dynamics with dense periodic points and dense $\mc{W}^{u/s}$ foliations are not possible on the M\"obius strip. Thus we have exhausted all possibilities for the topology of $M$ and are done.
\end{proof}

\bibliographystyle{alpha}
\bibliography{Refs}

\Addresses

\end{document}